\documentclass{amsart}

\usepackage{amsmath} 
\usepackage{amssymb}
\usepackage{cancel} 

\newtheorem{theorem}{Theorem}[section] 
\newtheorem{claim}{Claim}[theorem]
\newtheorem{lemma}[theorem]{Lemma} 
\newtheorem{proposition}[theorem]{Proposition} 
\newtheorem{observation}[theorem]{Observation} 
\newtheorem{corollary}[theorem]{Corollary} 

\theoremstyle{definition}
\newtheorem{definition}[theorem]{Definition}
\newtheorem{example}[theorem]{Example}
\newtheorem{problem}[theorem]{Problem}

\theoremstyle{remark}
\newtheorem{remark}[theorem]{Remark}

\newtheorem{context}[theorem]{Context}

\numberwithin{equation}{section}
\newcommand{\forces}{\Vdash}

\newcommand{\bV}{{\bf V}} 
\newcommand{\lesdot}{\mathrel{\mathord{<}\!\!\raise 
0.8 pt\hbox{$\scriptstyle\circ$}}}

\newcommand{\conc}{{}^\frown\!}
\newcommand{\lh}{{\rm lh}\/}
\newcommand{\rest}{{\restriction}}
\newcommand{\dom}{{\rm dom}} 
\newcommand{\rng}{{\rm rng}}

\newcommand{\cA}{{\mathcal A}}

\newcommand{\cB}{{\mathcal B}}

\newcommand{\gb}{{\mathfrak b}}

\newcommand{\bbD}{{\mathbb D}}
\newcommand{\cD}{{\mathcal D}}
\newcommand{\gd}{{\mathfrak d}} 

\newcommand{\cH}{{\mathcal H}}

\newcommand{\cF}{{\mathcal F}}
\newcommand{\cG}{{\mathcal G}}
\newcommand{\cI}{{\mathcal I}}

\newcommand{\bbP}{{\mathbb P}}

\newcommand{\bbQ}{{\mathbb Q}}

\newcommand{\cR}{{\mathcal R}}

\newcommand{\cS}{{\mathcal S}}
\newcommand{\cT}{{\mathcal T}}

\newcommand{\cX}{{\mathcal X}}

\newcommand{\cZ}{{\mathcal Z}}

\newcommand{\cf}{{\rm cf}\/} 
\newcommand{\cl}{{\rm cl}\/}

\newcommand{\hight}{{\rm ht}\/}

\newcommand{\st}{{\bf st}} 

\newcommand{\vare}{\varepsilon}

\newcommand{\Agame}{{\Game^{\rm rcA}_{d}}}

\newcommand{\pr}{{\rm pr}}

\newcommand{\suc}{{\rm succ}}
\newcommand{\mrot}{{\rm root}}
\newcommand{\spl}{{\rm split}}

\newcommand{\vtl}{\vartriangleleft}

\newcommand{\bairek}{{}^{\kappa}\kappa}
\newcommand{\locbh}{{\rm Loc}^{b,h}_\kappa}

\newcount\skewfactor
\def\mathunderaccent#1#2 {\let\theaccent#1\skewfactor#2
\mathpalette\putaccentunder}
\def\putaccentunder#1#2{\oalign{$#1#2$\crcr\hidewidth
\vbox to.2ex{\hbox{$#1\skew\skewfactor\theaccent{}$}\vss}\hidewidth}}
\def\name{\mathunderaccent\tilde-3 }

\begin{document}

\title{A case for clubs}

\author{Andrzej Ros{\l}anowski}
\address{Department of Mathematics\\
 University of Nebraska at Omaha\\
 Omaha, NE 68182-0243, USA}
\email{aroslanowski@unomaha.edu}


\subjclass{Primary: 03E40 Secondary: 03E35, 03E17}
\date{October 2026}

\begin{abstract}
  We present several remarks on cardinal coefficients associated with
  the generalized Baire space $\bairek$ for an inaccessible cardinal
  $\kappa$ and related forcing notions. We provide evidence that the
  appropriate coefficients to consider in $\bairek$ are those
  determined by binary relations restricted to a club. In other words,
  we make a case for the systematic use of clubs.
\end{abstract}

\maketitle

\bigskip \bigskip \bigskip

\section{Introduction}
Interest in special subsets of $\bairek$ (and related spaces) has been
steadily increasing since the beginning of this century. The article
\cite{KW20} gave a quick review of possible generalizations of the classical
concepts from the set theory of the reals. Additionally, Banach-Kuratowski
matrices and Lusin sets for $\bairek$ were studied in \cite{Ju18, Ju2x}
while strong measure zero subsets of ${}^\kappa 2$, the generalized 
Borel conjecture, and related notions of small sets were topics of
\cite{HkSh:662, ChSch25}. As is the case for the reals,
investigations of special sets are closely related to the study of cardinal
coefficients.  These coefficients were the focus of \cite{CuSh:541,MRSh:799,
  BTFFM17, BBFM-17, Br22, vdV2x, vdV25, BeSw0x} and numerous other works.

In the studies of cardinal coefficients related to $\bairek$ one has to look
for possible consistency results. The expected consistency arguments would
usually involve iterated forcing without collapsing cardinals, parallel to
FS iterations of ccc forcing notions or CS iterations of proper forcing
notions. However, it has been known since 1980s that relevant iterations
might collapse $\kappa^+$. A simple example of an $\omega$--step iteration
of ${<}\kappa$--complete $\kappa^+$--cc forcing notions with (full support)
limit collapsing $\kappa^+$ is given in \cite{Ro2x}. A general theory of
needed iterations is still under development, but many tools and case
studies are already presented in literature, both for iterations with
supports of sizes smaller than $\kappa$ and for iterations with
$\kappa$--support.

The advantage of ${<}\kappa$--support iterations is that seemingly we have
more control over them, including allowing $2^\kappa$ in the extension to be
arbitrarily large. The disadvantages of such iterations include adding
$\kappa$--Cohen reals at stages of cofinality $\kappa$. The
$\kappa$--support iterations may be somewhat more difficult to work with,
but there is a plentiful of general theory developed for them, see for
instance \cite{RoSh:655, RoSh:777, RoSh:860, RoSh:890, RoSh:942,
  RoSh:1001}.

In the present work, we investigate several cardinal characteristics of
$\bairek$ (where $\kappa$ is strongly inaccessible) and the natural forcing
notions associated with them. Our inquiry was motivated by the work of van
der Vlugt \cite{vdV25} and the many intriguing problems posed there. 
We concentrate on the use of forcing notions fitting the existing theory of  
$\kappa$--support iterations. A bad behaviour of the forcing notions related
to $\gb^b_\kappa(\leq^*)$ and $\gb^{b,h}_\kappa(\cancel{\ni^{\rm stat}})$
leads us to discovering combinatorial dependencies between corresponding
dominating numbers (Section 1, in particular Theorem \ref{re125two}). In
Section 2 we recall from \cite{RoSh:860, RoSh:1001} two properties that
imply $\kappa$--propernes of $\kappa$--support iterations. We also recall
the constructions of typical forcing notions with these properties. In the
following two sections, we examine the forcing notions (and relevant
consistency results) associated with $\leq^\cl$, $\cancel{\ni^{\rm stat}}$,
and $\in^*$, $\in^\cl$. Finally, in the fifth section, we note that there
appears to be no nice parallel to the iterations of the $k$--localization
property from \cite{NeRo93,Ro0x}.  \bigskip

\subsection{Cardinal charachteristics of $\bairek$}
For a binary relation $R\subseteq X\times Y$ we consider the {\em
  dominating number\/} $\gd(R)$ and the {\em bounding number\/}
$\gb(R)$ defined as follows: 
\begin{itemize}
\item $\gd(R)=\min\{|\cG|:\cG\subseteq Y\ \wedge\ (\forall x\in 
  X)(\exists g\in\cG)(x\;R\; g)\}$,
\item $\gb(R)=\min\{|\cF|:\cF\subseteq X\ \wedge\ \neg (\exists y\in  
  Y)(\forall f\in\cF)(f\;R\; y)\}$.
\end{itemize}
For functions $h\in \bairek$ and $b\in {}^\kappa(\kappa+1)$ such that
$h(\alpha)\leq b(\alpha)$ are cardinals (for all $\alpha<\kappa$) we
let $\locbh$ be the collection of all functions
$\varphi\in\prod\limits_{\alpha<\kappa} [b(\alpha)]^{<
  h(\alpha)}$. If $b$ is identically $\kappa$ then we denote it
$\bar{\kappa}$ and we write ${\rm Loc}^{\bar{\kappa},h}_\kappa$. Elements
of $\locbh$ are oftentimes called {\em slaloms of width $<h$ below
  $b$}.   

Cardinal coefficient related to the space $\bairek$ and studied most
are $\gb(R)$ and $\gd(R)$ for the following binary relations.
\begin{enumerate}
\item[$\leq^*$\quad ] $f\leq^* g$ if and only if ($f,g\in\bairek$ and)
  for some $\alpha<\kappa$, for all $\beta\geq \alpha$ we have
  $f(\beta) \leq g(\beta)$;
\item[$\cancel{=^\infty}$\quad] $f\; \cancel{=^\infty}\; g$ if and only if
  ($f,g\in\bairek$ and) for some $\alpha<\kappa$, for all $\beta\geq
  \alpha$ we have $f(\beta) \neq g(\beta)$;
\item[$\in^*$\quad] $f\in^*\varphi$ if and only if
  ($f\in\bairek$ and $\varphi\in {\rm Loc}^{\bar{\kappa},h}_\kappa$
  and) for some $\alpha<\kappa$, for all $\beta\geq 
  \alpha$ we have $f(\beta) \in \varphi(\beta)$;
\item[$\cancel{\ni^\infty}$\quad] $\varphi\;\cancel{\ni^\infty}\; f$ if and
  only if ($f\in\bairek$ and $\varphi\in {\rm Loc}^{\bar{\kappa},h}_\kappa$
  and) for some $\alpha<\kappa$, for all $\beta\geq 
  \alpha$ we have $f(\beta) \notin \varphi(\beta)$.
\end{enumerate}
These relations can be naturally restricted to bounded subspaces and
we may consider  $\leq^*\rest \big(\prod\limits_{\alpha<\kappa} b(\alpha)
\times \prod\limits_{\alpha<\kappa} b(\alpha)\big)$, $\cancel{=^\infty}
\rest \big(\prod\limits_{\alpha<\kappa} b(\alpha) \times 
\prod\limits_{\alpha<\kappa} b(\alpha)\big)$,  $\in^*\rest
\big(\prod\limits_{\alpha<\kappa} b(\alpha) \times\locbh\big)$ and
$\cancel{\ni^\infty} \rest \big(\locbh\times
\prod\limits_{\alpha<\kappa} b(\alpha)\big)$. We will use the same
symbols for the resticted relations but the fact that we are dealing
with the restrictions will be reflected in notation for cardinal
coefficients. Thus the dominating and bounding numbers for the
relations on $\bairek$ will be denoted, respectively, by
$\gd_\kappa(\leq^*)$,  $\gb_\kappa(\leq^*)$,
$\gd_\kappa(\cancel{=^\infty})$, $\gb_\kappa(\cancel{=^\infty})$,
$\gd_\kappa^h(\in^*)$, $\gb_\kappa^h(\in^*)$ and
$\gd^h_\kappa(\cancel{\ni^\infty})$, 
$\gb^h_\kappa(\cancel{\ni^\infty})$. The bounded versions of these
coefficients will be called $\gd_\kappa^b(\leq^*)$,
$\gb_\kappa^b(\leq^*)$, $\gd_\kappa^b(\cancel{=^\infty})$,
$\gb_\kappa^b(\cancel{=^\infty})$, $\gd_\kappa^{b,h}(\in^*)$,
$\gb_\kappa^{b,h}(\in^*)$ and 
$\gd^{b,h}_\kappa(\cancel{\ni^\infty})$,
$\gb^{b,h}_\kappa(\cancel{\ni^\infty})$.

Previous studies of $\kappa$--properness for $\kappa$--support
iterations suggest that instead of considering relations that hold
``from some point on,'' we can focus on relations that hold ``on a club
of $\alpha<\kappa$.'' This gives us the following binary relations:
\begin{enumerate}
\item[$\leq^\cl$\quad ] $f\leq^\cl g$ if and only if ($f,g\in\bairek$ and)
 the set $\big\{\beta<\kappa:f(\beta) \leq g(\beta)\big\}$ contains a
 club of $\kappa$;
\item[$\cancel{=^{\rm stat}}$\quad] $f\; \cancel{=^{\rm stat}}\; g$ if
  and only if ($f,g\in\bairek$ and) the set $\big\{\beta<\kappa:
  f(\beta) \neq g(\beta)\big\}$ contains a club of $\kappa$; 
\item[$\in^\cl$\quad] $f\in^\cl\varphi$ if and only if
  ($f\in\bairek$ and $\varphi\in {\rm Loc}^{\bar{\kappa},h}_\kappa$
  and) the set $\big\{\beta<\kappa: f(\beta) \in \varphi(\beta)\big\}$
  contains a club of $\kappa$;
\item[$\cancel{\ni^{\rm stat}}$\quad] $\varphi\;\cancel{\ni^{\rm stat}}\; f$ if and
  only if ($f\in\bairek$ and $\varphi\in {\rm Loc}^{\bar{\kappa},h}_\kappa$
  and) the set $\big\{\beta<\kappa: f(\beta) \notin
  \varphi(\beta)\big\}$ contains a club of $\kappa$,
\end{enumerate}
and similarly with the versions ``below $b\in\bairek$.'' The
corresponding cardinal coefficients will be denoted naturally, so we
will have $\gb_\kappa^b(\leq^\cl)$, $\gd_\kappa^b(\cancel{=^{\rm
    stat}})$, etc.

Cardinal coefficients associated with relations ``on a club'' were
considered in several papers in the past. For instance in 1995,
Cummings and Shelah \cite{CuSh:541} proved that, for a  strongly
inaccessible $\kappa$, $\gd_\kappa(\leq^*)=\gd_\kappa(\leq^\cl)$ and
$\gb_\kappa(\leq^*)=\gb_\kappa(\leq^\cl)$. In contrast, recently Benhamou
and Switzer \cite{BeSw0x} showed that for increasing $h,h'\in\bairek$ it is
consistent that $\gb_h(\in^*)<\gb_{h'}(\in^\cl)$ and it is also consistent
that $\gd_{h'}(\in^\cl)<\gd_h(\in^*)$, 
\medskip

\subsection{Notation}
Our notation is rather standard and compatible with that of classical
textbooks (like Jech \cite{J}). In forcing however we keep the older
convention that {\em a stronger condition is the larger one}.  

\begin{enumerate}
\item Ordinal numbers will be denoted by the lower case initial letters of
the Greek alphabet ($\alpha,\beta,\gamma,\delta,\vare,\zeta$) and also
by $\xi, i,j$ (with possible sub- and superscripts). Cardinal numbers will
be called $\kappa,\lambda,\mu$; {\em $\kappa$ will be  always assumed to be  
  strongly inaccessible\/}. 

By $\chi$ we will denote a {\em sufficiently large\/} regular cardinal; 
$\cH(\chi)$ is the family of all sets hereditarily of size less than
$\chi$. Moreover, we fix a well ordering $<^*_\chi$ of $\cH(\chi)$. 

\item For two sequences $\eta,\nu$ we write $\nu\vtl\eta$
whenever $\nu$ is a proper initial segment of $\eta$, and $\nu
\trianglelefteq\eta$ when either $\nu\vtl\eta$ or $\nu=\eta$. 
The length of a sequence $\eta$ is denoted by $\lh(\eta)$.

\item {\em A tree\/} is a $\vtl$--downward closed set of sequences. It $T$
  is a tree and $t\in T$, then $\suc_T(t)=\{x:t\conc \langle x\rangle\in
  T\}$, $\spl(T)=\{s\in T:|\suc_T(s)|\geq 2\}$ and $\mrot(T)$ is the
  shortest member of $\spl(T)$ (assuming $\spl(T)\neq\emptyset$). We also
  define $(T)_t=\{s\in T:s\trianglelefteq t\mbox{ or }t\vtl s\}$. A tree $T$
  is a {\em complete tree\/} if every $\vtl$--chain of length less than its
  height ${\rm ht}(T)$ of elements of $T$ has a $\vtl$--upper bound in $T$.
  If the height ${\rm ht}(T)$ is $\kappa$ then we may say {\em complete
    $\kappa$--tree}. 
\item For a complete $\kappa$--tree $T$, the set of all limit
  $\kappa$--branches through $T$ is denoted $\lim(T)$. Thus $f\in\lim(T)$ if
  and only if $f$ is a function on $\kappa$ and $f\rest\alpha\in T$ for all
  $\alpha<\kappa$.  
\item We will consider several games of two players. One player will be
called {\em Generic\/} or {\em Complete\/}, and we will refer to this player
as ``she''. Her opponent will be called {\em Antigeneric\/} or {\em
  Incomplete} and will be referred to as ``he''.  

\item For a forcing notion $\bbP$, all $\bbP$--names for objects in the
  extension via $\bbP$ will be denoted with a tilde below (e.g.,
  $\name{\tau}$, $\name{X}$) and $\name{G}=\name{G}_\bbP$ stands for the
  canonical  $\bbP$--name for the generic filter in $\bbP$.

By ``$\kappa$--support iterations'' we mean iterations in which domains of
conditions are of size $\leq\kappa$.
\end{enumerate}
\bigskip

\begin{definition}
\label{strcom}
Let $\bbP$ be a forcing notion.
\begin{enumerate}
\item Let $\Game_0^\kappa(\bbP)$ be the following game of two players, {\em
    Complete} and  {\em Incomplete}:
  \begin{quotation}
\noindent the game lasts at most $\kappa$ moves and during a play the
players construct a sequence $\langle (p_i,q_i): i<\kappa\rangle$ of pairs
of conditions from $\bbP$ in such a way that $(\forall j<i<\kappa)(
p_j\leq q_j\leq p_i)$ and at the stage $i<\kappa$ of the game, first 
Incomplete chooses $p_i$ and then Complete chooses $q_i$.  
\end{quotation}
Complete wins if and only if for every $i<\kappa$ there are legal moves for
both players. 
\item We say that the forcing notion $\bbP$ is {\em strategically
$({<}\kappa)$--complete\/} if Complete has a winning strategy in the game 
$\Game_0^\kappa(\bbP)$.
\item Let $N\prec (\cH(\chi),\in,<^*_\chi)$ be a model such that
${}^{<\kappa} N\subseteq N$, $|N|=\kappa$ and $\bbP\in N$. We say that a
condition $p\in\bbP$ is {\em $(N,\bbP)$--generic in the standard
sense\/} (or just: {\em $(N,\bbP)$--generic\/}) if for every
$\bbP$--name $\name{\tau}\in N$ for an ordinal we have $p\forces$``
$\name{\tau}\in N$ ''. 
\item $\bbP$ is {\em $\kappa$--proper in the standard sense\/} (or just:
{\em $\kappa$--proper\/}) if there is $x\in \cH(\chi)$  such that for
every model $N\prec (\cH(\chi),\in,<^*_\chi)$ satisfying  
\[{}^{<\kappa} N\subseteq N,\quad |N|=\kappa\quad\mbox{ and }\quad\bbP,x
  \in N, \]
and for every condition $q\in N\cap\bbP$ there is an $(N,\bbP)$--generic
condition $p\in\bbP$ stronger than $q$.
\end{enumerate}
\end{definition}

In this paper we assume the following.

 \begin{context}
\label{incon}
\begin{enumerate}
\item[(a)] $\kappa$ is a strongly inaccessible cardinal and $\bar{\kappa}$
  denotes a function on $\kappa$ with the constant value $\kappa$.
\end{enumerate}
Parameters $b,h$ etc for our cardinal coefficients will be usually either
$\bar{\kappa}$ or they will belong to one of the following three sets of
functions.
\begin{enumerate}
\item[(b)] $\cR^-_\kappa$ is the family of all non-decreasing functions
  $b\in {}^\kappa\kappa$ such that  $b(\alpha)\geq 3$ is a cardinal for
  every $\alpha<\kappa$.
\item[(c)]  $\cR_\kappa$ is the family of those $b\in\cR^-_\kappa$ that
  each value $b(\alpha)$ is regular infinite cardinal and the set
  $\{\alpha<\kappa: \alpha<b(\alpha)\}$ contains a club subset of $\kappa$ 
\item[(d)]  $\cR^+_\kappa$ consists of all strictly increasing $b\in
  \cR_\kappa$ such that $\alpha<b(\alpha)$ for each $\alpha<\kappa$. 
\end{enumerate}
\end{context}

\section{Around $\gb^b_\kappa(\leq^*)$ and
  $\gb^{b,h}_\kappa(\cancel{\ni^\infty})$}  
In \cite[Question 125]{vdV25} Tristan van der Vlugt asked:
\begin{itemize}
\item Is it consistent that $\gd_\kappa^b(\leq^*)<\gd_\kappa(\leq^*)$ for all 
$b\in \bairek$?
\item Is it consistent that $\gb_\kappa(\leq^*)<\gb_\kappa^b(\leq^*)$ for all 
$b\in \bairek$?
\end{itemize}
The second part of \cite[Question 129]{vdV25} asked a related question if
there exist nontrivial $b,h,b',h'\in\bairek$ such that, consistently,
$\gb^{b',h'}_\kappa(\cancel{\ni^\infty})
<\gb^{b,h}_\kappa(\cancel{\ni^\infty})$. The two questions are somewhat
related because  
$\gb^b_\kappa(\leq^*) = \gb^{b,b}_\kappa(\cancel{\ni^\infty})$. Also the
natural forcing notions related to these problems are very similar.

In this section we will present a couple of observations suggesting that the
answer to these questions may require the use of substantially larger cardinals.

By \cite[Lemma 17]{vdV25}, the cardinal $\gb_\kappa^b(\leq^*)$ can
  have value bigger than $\kappa$ only if there exists a club set
  $C\subseteq \kappa$ such that for each $\xi\in C$ we have
  $\cf(b(\alpha)) > \xi$ for all $\alpha\geq \xi$. Also,
  $\gb_\kappa^b(\leq^*)$ will not change if we replace each $b(\alpha)$
  with $\cf\big(b(\alpha)\big)$. Therefore, when investigating
  $\gb_\kappa^b(\leq^*)$  it is only natural to restrict our attention to
  $b\in \cR_\kappa$. Similarly, in the studies of $\gb^{b,h}_\kappa
  (\cancel{\ni^\infty})$, \cite[Lemma 49]{vdV25} suggests we should assume
  $h\leq b\in\cR_\kappa\cup\{\bar{\kappa}\}$ and the strategic completeness
  of  the relevant forcing will require
  $h\in\cR_\kappa\cup\{\bar{\kappa}\}$.  

If we are interested in forcing a model for $\gb_\kappa(\leq^*)<
\gb^b_\kappa(\leq^*)$ we may try to iterate with $\kappa$--support forcing
notions adding a dominating function below $b$ but without adding a
$\bairek$--dominating real. A natural candidate to be used here is
$\bbQ^*_b$ defined below. Similarly, for a model for $\gb^{b',h'}_\kappa
(\cancel{\ni^\infty})<\gb^{b,h}_\kappa(\cancel{\ni^\infty})$ we may be
thinking about iterating with $\kappa$--supports forcing notions adding a
function below $b$ avoiding slaloms from ${\rm Loc}^{b,h}_\kappa\cap
\bV$. The candidate for it would be $\bbQ^h_b$.

\begin{definition}
  [See {\cite[Definition 2.20(1)]{RoSh:1001}}]
\label{boudQ4}
Assume that $\kappa$ is inaccessible and $h,b\in
\cR_\kappa\cup\{\bar{\kappa}\}$ and $h\leq b$. We define forcing notions
$\bbQ^*_b$ and $\bbQ^h_b$ as follows.

\begin{enumerate}
\item {\bf A condition} in $\bbQ^*_b$ is a complete $\kappa$--tree
  $T\subseteq \bigcup\limits_{\alpha<\kappa}
  \prod\limits_{\beta<\alpha}b(\beta)$ such that  for every $t\in T$, if 
  $\mrot(T)\trianglelefteq t$ then $\suc_T(t)$ is a co-bounded subset of
  $b\big(\lh(t)\big)$. 

\noindent  
{\bf The order} of $\bbQ^*_b$ is the reverse inclusion.
\medskip

\noindent  
We will say that a condition $U\in \bbQ^*_b$ {\em purely extends\/}
$T\in\bbQ^*_b$ ($U\geq_\pr T$ in short) if $U\geq T$ and
$\mrot(U)=\mrot(T)$.

\item {\bf A condition} in $\bbQ^h_b$ is a complete $\kappa$--tree
  $T\subseteq \bigcup\limits_{\alpha<\kappa}
  \prod\limits_{\beta<\alpha}b(\beta)$ such that  for every $t\in T$, if 
  $\mrot(T)\trianglelefteq t$ then $\big|b(\lh(t))\setminus
  \suc_T(t)\big|<h(\alpha)$.
  
\noindent  
{\bf The order} of $\bbQ^h_b$ is the reverse inclusion.
\medskip

\noindent  
We will say that a condition $U\in \bbQ^h_b$ {\em purely extends\/}
$T\in\bbQ^h_b$ ($U\geq_\pr T$ in short) if $U\geq T$ and
$\mrot(U)=\mrot(T)$.
\end{enumerate}
\end{definition}

The forcings $\bbQ^*_b$ are a special type of $\bbQ^h_b$:
$\bbQ^*_b=\bbQ^b_b$ for $b\in \cR_\kappa\cup\{\bar{\kappa}\}$. 

Clearly, the forcing notion $\bbQ^h_b$ adds an element of
$\prod\limits_{\alpha<\kappa}b(\alpha)$ avoiding (in the sense of
$\cancel{\ni}^\infty$) all slaloms from  ${\rm Loc}^{b,h}_\kappa\cap
\bV$. If $b=h$ this means that $\bbQ^*_b$ adds a $\leq^*$--dominating 
function in $(\prod\limits_{\alpha<\kappa}b(\alpha),\leq^*)$.

By \cite[Proposition B.6.2]{RoSh:777}, $\bbQ^h_b$ is strategically
$({<}\kappa)$--complete (here we need to use the assumption that
$h\in\cR_\kappa$ or $h=\bar{\kappa}$); this is also shown in Proposition
\ref{q2isAb}(1) here. By \cite[Proposition 3.9(d),
Theorem 4.1]{RoSh:1001} (or Theorem \ref{n3.7} here), $\kappa$--support
iterations of these forcings will be $\kappa$--proper. But even though
conditions in $\bbQ^h_b$  look like bounded trees, the forcing $\bbQ^h_b$ is
not ${}^\kappa\kappa$--bounding.       

\begin{proposition}
  \label{addunb}
  Assume that $\kappa$ is inaccessible and $h,b\in
\cR_\kappa\cup\{\bar{\kappa}\}$ and $h\leq b$. Then the forcing notion
$\bbQ^h_b$ adds a $\leq^*$--unbounded function in ${}^\kappa\kappa$.  
\end{proposition}

\begin{proof}
  For each $\alpha<\kappa$ choose $A_\alpha\subseteq b(\alpha)$ such that
  $|A_\alpha| =|b(\alpha)\setminus A_\alpha|=b(\alpha)$. For $x\in
  \prod\limits_{\alpha<\kappa} b(\alpha)$  let $S_x=\{\alpha<\kappa:
  x(\alpha)\in A_\alpha\}$ and let $g_x(\alpha)=\min\big(S_x\setminus
  (\alpha+1)\big)$ (with the convention that $\min(\emptyset)=0$). Let
  $\name{W}$ be the canonical  $\bbQ^h_b$--name for the generic function in
  $\prod\limits_{\alpha<\kappa}   b(\alpha)$. Thus
  $T\forces_{\bbQ^h_b} \mrot(T)\vtl \name{W}\in \prod\limits_{\alpha<\kappa}
  b(\alpha)$. An easy density argument shows that for each $f\in \bairek$, 
  \[\forces_{\bbQ^h_b} \big(\forall \xi<\kappa\big)\big( \exists
    \alpha>\xi\big)\big(\forall \beta\in [\alpha,\alpha+f(\alpha)+1)\big)
    \big(\name{W}(\beta)\notin A_\beta\ \wedge\ \name{W}(\alpha+f(\alpha)+1
    )\in  A_{\alpha+f(\alpha)+1}\big).\]
  Consequently, $\forces_{\bbQ^h_b}\big(\forall f\in {}^\kappa\kappa\cap
   \bV\big)\big(\forall \xi<\kappa\big)\big(\exists \alpha>\xi\big)
   \big(f(\alpha)<g_{\name{W}}(\alpha)\big)$.
 \end{proof}
 
 The above proposition actually shows that $\bbQ^h_b$ adds a $\kappa$--Cohen
 real. This does not immediately eliminate the natural proof for
the consistency of  $\gb_\kappa(\leq^*)<\gb^b_\kappa(\leq^*)$, because we
still may hope for the iterations of $\bbQ^*_b=\bbQ^b_b$ not adding {\em
  dominating\/} functions. However, for a small inaccessible $\kappa$ with
diamonds, the forcing notion $\bbQ^h_b$ adds a $\leq^*$--dominating
element of  ${}^\kappa\kappa$.  

\begin{proposition}
  \label{bad}
Suppose $\kappa$ is a strongly inaccessible not Mahlo cardinal and that
$\diamondsuit_\kappa$ holds true. Let $h,b\in\cR_\kappa$ be such that $h\leq
b$. Then the forcing notion $\bbQ^h_b$ adds a ${}^\kappa\kappa$--dominating
function. That is, for some $\bbQ^h_b$--name $\name{\tau}$ we have
\[\forces_{\bbQ^h_b} \mbox{`` }\name{\tau}\in {}^\kappa\kappa\ \wedge\
  (\forall x\in {}^\kappa\kappa\cap \bV)(\exists \alpha<\kappa)(\forall
  \beta\geq \alpha)(x(\beta)\leq \name{\tau}(\beta))\mbox{''.}\]
\end{proposition}

\begin{proof}
  Let $\kappa,h,b$ be as in the assumptions and let $C\subseteq \kappa$ be a 
  club consisting of singular cardinals only and such that
  $\alpha<h(\alpha)$ for $\alpha\in C$. Let $\langle f_\delta:\delta\in
  C\rangle$ be such that
  \begin{enumerate}
\item[$(*)_1$] $f_\delta\in \prod\limits_{\alpha<\delta} b(\alpha)$ for each
    $\delta\in C$, and
\item[$(*)_2$] for each $f\in \prod\limits_{\alpha<\kappa} b(\alpha)$ the set
  $\{\delta\in C: f_\delta\subseteq f\}$ is stationary. 
\end{enumerate}
(So this is an ``interpretation'' of a diamond sequence on $\kappa$.)

\begin{claim}
  \label{cl1}
For each $S\in\bbQ^h_b$ and $\alpha<\kappa$ we have $\{f_\delta: \delta\in
C\setminus \alpha\}\cap S\neq \emptyset$. 
\end{claim}

\begin{proof}[Proof of the Claim]
  Pick $f\in \lim(S)$. There is $\vare\in C$ above
  $\max\big(\alpha,\lh(\mrot(S))\big)$ such that $f_\vare=f\rest \vare\in
  \{f_\delta:\delta\in C\}\cap S$.  
\end{proof}

\begin{claim}
  \label{cl2}
Suppose $\alpha<\beta$ are from $C$. Then there is a function $\psi\in
\prod\limits_{\xi\in [\alpha,\beta)}\big[b(\xi)\big]^{\textstyle {<}h(\xi)}$ such that
\[\big(\forall \delta\in C\cap (\alpha,\beta]\big) \big(\exists\xi\in
  [\alpha,\delta) \big) \big(f_\delta(\xi)\in \psi(\xi) \big).\]
\end{claim}

\begin{proof}[Proof of the Claim]
Induction on the order type of $C\cap (\alpha,\beta]$.

If $\beta=\min \big(C\setminus (\alpha+1) \big)$, then picking any
$\psi\in\prod\limits_{\xi\in [\alpha,\beta)} \big[b(\xi)\big]^{\textstyle
  {<}h(\xi)}$ with $f_\beta(\alpha)\in \psi(\alpha)$ will do. 
Similarly, if for some $\gamma\in C\cap (\alpha,\beta)$ we have
$\beta=\min\big(C\setminus (\gamma+1) \big)$, then we apply the inductive 
hypothesis to $C\cap (\alpha,\gamma]$ to find a function
$\psi_0\in\prod\limits_{\xi\in [\alpha,\gamma)} \big[b(\xi)\big]^{\textstyle
  {<}h(\xi)} $ such that  
\[\big(\forall \delta\in C\cap (\alpha,\gamma]\big) \big(\exists\xi\in
  [\alpha,\delta) \big) \big(f_\delta(\xi)\in \psi_0(\xi) \big).\]
 Then we extend $\psi_0$ to $\psi\in\prod\limits_{\xi\in [\alpha,\beta)}
 \big[b(\xi)\big]^{\textstyle{<}h(\xi)}$ by letting $\psi(\xi)=
 \{f_\beta(\xi)\}$ for $\xi\in [\gamma,\beta)$. 

 Now suppose that $\beta=\sup(C\cap \beta)$. We know that
 $\cf(\beta)<\beta$, so we may choose an increasing continuous sequence
 $\langle \gamma_i:i<\cf(\beta)\rangle\subseteq C\cap (\alpha,\beta)$ with
 $\cf(\beta)<\gamma_0$ and $\sup\limits_{i<\cf(\beta)} \gamma_i=\beta$. The
 order type of each of the sets $C\cap (\alpha,\gamma_0]$ and $C\cap
 (\gamma_i,\gamma_{i+1}]$ for $i<\cf(\beta)$ is smaller than the order type
 of $C\cap (\alpha,\beta]$. By the inductive hypothesis we may choose
 functions $\psi_0\in\prod\limits_{\xi\in [\alpha,\gamma_0)}
 \big[b(\xi)\big]^{\textstyle {<}h(\xi)}$ and $\psi_{i+1}  \in
 \prod\limits_{\xi\in [\gamma_i,\gamma_{i+1})} \big[b(\xi)\big]^{\textstyle 
  {<}h(\xi)}$ (for $i<\cf(\beta)$) such that
 \[\begin{array}{l}
   \big(\forall \delta\in C\cap (\alpha,\gamma_0]\big) \big(\exists\xi\in
   [\alpha,\delta) \big) \big(f_\delta(\xi)\in \psi_0(\xi) \big)\quad \mbox{
    and}\\
\ \\
     \big(\forall \delta\in C\cap (\gamma_i,\gamma_{i+1}]\big)
     \big(\exists\xi\in  [\gamma_i,\delta) \big) \big(f_\delta(\xi)\in 
     \psi_{i+1}(\xi) \big)\quad\mbox{  for  all }i<\cf(\beta). 
   \end{array}\]
Let $Z=\{f_{\gamma_i}(\gamma_0):0<i<\cf(\beta)\}\cup\{f_\beta(\gamma_0)\}
\cup\psi_1( \gamma_0)$ (note that $\cf(\beta)<\gamma_0<h(\gamma_0)$ so
$|Z|<h(\gamma_0)$). Define  $\psi\in\prod\limits_{\xi\in [\alpha,\beta)}
\big[b(\xi)\big]^{\textstyle  {<}h(\xi)}$ by
  \begin{quotation}
\noindent  $\psi\rest [\alpha,\gamma_0)=\psi_0$,\ \  $\psi(\gamma_0)=Z$,  
 $\psi\rest (\gamma_0,\gamma_1)=\psi_1\rest (\gamma_0,\gamma_1)$ and 

\noindent $\psi\rest [\gamma_i,\gamma_{i+1})=\psi_{i+1}$ for
    $0<i<\cf(\beta)$. 
  \end{quotation}
Then $\psi$ will have the desired property.  
\end{proof}
\bigskip

Let $\forces_{\bbQ^h_b}\mbox{`` }\name{A}=\big\{\delta\in C:\big(\exists
T\in \name{G}\big) \big(f_\delta\subseteq \mrot(T)\big)\big\}\mbox{ ''}$. By 
Claim \ref{cl1} we know that $\forces |\name{A}|=\kappa$. Let
$\name{\tau}_0$ be a $\bbQ^h_b$--name for the increasing enumeration of
$\name{A}$ and let $\name{\tau}$ be such that
\[\forces_{\bbQ^h_b}\mbox{`` }\name{\tau}\in {}^\kappa\kappa\ \wedge\
  \big(\forall \xi<\kappa\big)\big(\name{\tau}(\xi)=
  \name{\tau}_0(\xi+1)\big) \mbox{ ''}\]

\begin{claim}
  \label{cl3}
  $\forces_{\bbQ^h_b}\mbox{`` }\big(\forall x\in {}^\kappa\kappa\cap\bV\big) 
  \big(\exists\alpha< \kappa\big)\big(\forall\xi>\alpha\big)\big(x(\xi)\leq
  \name{\tau}(\xi)\big)  \mbox{ ''}$ 
\end{claim}

\begin{proof}[Proof of the Claim]
Let $x\in {}^\kappa\kappa$ and let $S\in\bbQ^h_b$. Pick a strictly
increasing continuous sequence $\langle\gamma_i:i<\kappa\rangle\subseteq C$ 
such that $\lh(\mrot(S))<\gamma_0$ and $\big(\forall i<\kappa\big)
\big(\forall\xi< \gamma_i\big)\big(x(\xi)<\gamma_i\big)$. Use Claim
\ref{cl2} to build a function $\psi\in {\rm Loc}^{b,h}_\kappa$ such that 
\begin{enumerate}
\item[$(*)_3$]  $\big(\forall i<\kappa\big)\big(\forall\delta\in C\cap
  (\gamma_i,\gamma_{i+1}]\big) \big(\exists \xi\in [\gamma_i,\delta)\big)
  \big(f_\delta(\xi)\in \psi(\xi)\big)$.  
\end{enumerate}
Let $T\geq S$ be a condition from $\bbQ^h_b$ such that
\begin{enumerate}
\item[$(*)_4$] $\mrot(S)=\mrot(T)$ and
\item[$(*)_5$] if $t\in T$ then ($t\in S$ and) $\suc_T(t)=\suc_S(t)
  \setminus \psi\big(\lh(t)\big)$.  
\end{enumerate}
Then 
\[T\forces_{\bbQ^h_b} \mbox{`` }\name{A}\setminus (\gamma_0+1)\subseteq
  \big\{\gamma_j: j<\kappa\mbox{ is limit }\big\} \mbox{ ''}.\]
Consequently, for $\gamma_i\leq \xi<\gamma_{i+1}$, $1\leq i<\kappa$, we have 
\[T\forces_{\bbQ^h_b} \mbox{`` }\xi\leq \name{\tau}_0(\xi)=\gamma_j\mbox{
    for some  limit }j<\kappa\mbox{ ''}.\]
Hence $T\forces_{\bbQ^h_b} \mbox{`` }\name{\tau}(\xi) = \name{\tau}_0(\xi+1)
\geq \gamma_{i+1}\mbox{ ''}$. Thus $T\forces_{\bbQ^h_b} \mbox{``
}x(\xi)<\gamma_{i+1}\leq \name{\tau}(\xi)\mbox{ ''}$. Consequently,
$T\forces_{\bbQ^h_b} \mbox{`` }\big(\forall \xi\geq
  \gamma_1\big)\big(x(\xi)<\name{\tau}(\xi)\big)\mbox{ ''}$.
\end{proof}
\end{proof}
\medskip

The assumptions on $\kappa$ in Proposition \ref{bad}
can be possibly weaken, but there are limits on how far.

\begin{proposition}
  \label{wcomp}
Assume $\kappa$ is weakly compact, $b,h\in\cR_\kappa$ and $h\leq b$. Then
the forcing notion $\bbQ^h_b$ does not add any $\leq^*$--dominating function
in ${}^\kappa\kappa$.    
\end{proposition}

\begin{proof}
Let   $\cX=\bigcup\limits_{\alpha<\kappa} \prod\limits_{\beta<\alpha} 
  b(\beta)$. We say that a set $\cA\subseteq \cX$ is {\em
    $\bbQ^h_b$--dense\/} if $T\cap \cA\neq \emptyset$ for all
  $T\in \bbQ^h_b$. Note that if $\name{\tau}$ is a $\bbQ^h_b$--name for an
  ordinal and $\cA$ consists of all $t\in\cX$ such that some condition
  $T\in \bbQ^h_b$ with $\mrot(T)=t$ decides the value of $\name{\tau}$, then
  $\cA$ is $\bbQ^h_b$--dense.

  \begin{claim}
    \label{cl4}
Suppose $\cA\subseteq\cX$ is $\bbQ^h_b$--dense and $S\in \bbQ^h_b$. Then
there is $\cB\in [\cA]^{\textstyle {<}\kappa}$ such that for every
$T\geq_\pr S$, $T\cap \cB\neq \emptyset$. 
  \end{claim}

  \begin{proof}[Proof of the Claim]
For $S\in\bbQ^h_b$ and $\xi<\kappa$ let $S\rest\xi=\{s\rest\xi: s\in S\}$
(so this is a tree of height $\xi$). Let $\cT$ consist of all trees
$T\rest\xi$ such that $T\geq_\pr S$ and $T\rest \xi\cap\cA=\emptyset$.  

Equipped with the end-extension relation, $\cT$ forms a tree without any
$\kappa$--branch but with levels of size $<\kappa$. Therefore, by the weak
compactness of $\kappa$, the height $\hight(\cT)$ of $\cT$ must be smaller than
$\kappa$.  Put $\cB=\{t\in\cA:\lh(t)\leq \hight(\cT)+1\}$.  
  \end{proof}

 Now suppose $\name{f}$ is a $\bbQ^h_b$--name for an element of
 ${}^\kappa\kappa$. For $\xi<\kappa$  let $\cA_\xi$ be the set of all
 $t\in\cX$ such that there is a condition with root $t$ deciding the value
 of $\name{f}(\xi)$. Use Claim \ref{cl4} to choose inductively $\cB_\xi\in
 [\cA_\xi]^{\textstyle{<}\kappa}$ and $\alpha_\xi<\kappa$ so that
 \begin{enumerate}
 \item[$(\otimes)_1$] $\langle \alpha_\xi:\xi<\kappa\rangle$ is continuously increasing,
 \item[$(\otimes)_2$] $\cB_\xi\subseteq \cA_\xi\cap
   \bigcup\limits_{\alpha<\alpha_{\xi+1}} \prod\limits_{\beta<\alpha} b(\beta)$,
 \item[$(\otimes)_3$]  if $t\in \prod\limits_{\beta<\alpha_\xi} b(\beta)$
   and $S\in \bbQ^h_b$, $\mrot(S)=t$, then $\cB_\xi\cap S\neq \emptyset$.
 \end{enumerate}
Let $g\in {}^\kappa\kappa$ be such that for each $\xi<\kappa$ and $t\in
\cB_\xi$, if $S\in \bbQ^h_b$ satisfies $\mrot(S)=t$ and $S\forces
\name{f}(\xi)=\gamma$, then $\gamma<g(\xi)$. It should be clear by our
choices that

$\forces_{\bbQ^h_b}\mbox{`` } \big(\forall\alpha<\kappa\big)\big(
  \exists\beta> \alpha\big)\big(\name{f}(\beta)<g(\beta)\big)\mbox{ ''.}$
\end{proof}

The proof of Proposition \ref{bad} contains an interesting combinatorial fact.

\begin{theorem}
  \label{re125two}
Suppose $\kappa$ is strongly inaccessible not Mahlo and
$\diamondsuit_\kappa$ holds true. Let $b\in\cR^+_\kappa$. Then
$\gd_\kappa(\leq^*)\leq\gd^{b,{\rm id}^+}_\kappa(\cancel{\ni^\infty}) \leq
\gd^{b,b}_\kappa(\cancel{\ni^\infty})=\gd_\kappa^b(\leq^*)$.  
\end{theorem}

\begin{proof}
We will show the first inequality, the rest is trivial.  Let $C\subseteq
\kappa$ be a club consisting of singular cardinals only. Let $\langle
f_\delta:\delta\in C\rangle$ be such that 
  \begin{enumerate}
\item[$(*)_1$] $f_\delta\in \prod\limits_{\alpha<\delta} b(\alpha)$ for each
    $\delta\in C$, and
\item[$(*)_2$] for each $f\in \prod\limits_{\alpha<\kappa} b(\alpha)$ the set
  $\{\delta\in C: f_\delta\subseteq f\}$ is stationary. 
\end{enumerate}
Since $\kappa^{<\kappa}=\kappa$ and $\gd^{b,{\rm
    id}^+}_\kappa(\cancel{\ni^\infty}) \geq\kappa^+$, we may fix a
$\cancel{\ni}$--dominating (sic!) family $\cZ\subseteq \prod\limits_{\alpha<\kappa}
b(\alpha)$ of size $\gd^{b,{\rm id}^+}_\kappa(\cancel{\ni^\infty})$.

For $f\in\cZ$ let $x_f\in {}^\kappa\kappa$ be such that $x_f(\xi)$ is
the $(\xi+1)^{\rm st}$ element of the set $\{\delta\in C:f\rest\delta
=f_\delta\}$. We will argue that $\{x_f:f\in\cZ\}$ is a
$\leq^*$--dominating family in ${}^\kappa\kappa$.

Given $x\in {}^\kappa\kappa$. Pick a strictly increasing continuous sequence
$\langle\gamma_i:i<\kappa\rangle\subseteq C$ such that $\big(\forall
i<\kappa\big) \big(\forall\xi< \gamma_i\big)\big(x(\xi)< \gamma_i\big)$. Use
Claim \ref{cl2} to build a function $\psi\in\prod\limits_{\xi<\kappa}
\big[b(\xi)\big]^{\textstyle {\leq} |\xi|}$ such that 
\[\big(\forall i<\kappa\big)\big(\forall\delta\in C\cap
  (\gamma_i,\gamma_{i+1}]\big) \big(\exists \xi\in [\gamma_i,\delta)\big)
  \big(f_\delta(\xi)\in\psi(\xi)\big).\]  
Let $f\in\cZ$ be $\cancel{\ni}$--above $\psi$, that is $(\forall\xi<\kappa)(
f(\xi)\notin\psi(x))$. Then
\[\{\delta\in C\setminus (\gamma_0+1): f\rest\delta=f_\delta\}\subseteq
  \{\gamma_j: j<\kappa\mbox{ is limit }\}.\]
Consequently, for $\gamma_i\leq \xi<\gamma_{i+1}$, $1\leq i<\kappa$, the
value of $x_f(\xi)$ is $\gamma_j$ for some limit $j>i$. Hence
$x(\xi)<\gamma_{i+1}<x_f(\xi)$. 
\end{proof}

Unpublished old result of Woodin showed that, consistently, diamond may
fail at inaccessible cardinals. This result was improved and elaborated by
several authors, for instance Golshani \cite{Gol22} showed that the failure
of diamond may happen at the first inaccessible cardinal. However those
arguments require starting with substantially larger cardinal (\cite{Gol22}
starts with $\kappa$ which is $(\kappa+3)$--strong). Therefore the results
of the current section suggest that for any consistency results required by
\cite[Question 125]{vdV25} we may have to start with larger cardinals than
just inaccessible.  

\begin{problem}
  \label{prob1}
  \begin{enumerate}
  \item  Is it consistent that $\kappa$ is the first inaccessible cardinal and 
    $\gd_\kappa^b(\leq^*)<2^\kappa$ for all $b\in\cR_\kappa^+$?
  \item Is it consistent  for the first inaccessible $\kappa$, that 
 $\gd_\kappa^{b,h}(\cancel{\ni^\infty})<\gd_\kappa(\leq^*)$ for all (or some) 
 $b,h\in\cR_\kappa^+$?     
  \end{enumerate}
  \end{problem}

\section{Interlude: Forcings and Properness}
In this section first we recall two properties usefull in $\kappa$--support
iterations: {\em reasonable $A$--boundedness\/} of \cite{RoSh:860,RoSh:890}
and {\em purely sequential properness\/} of \cite{RoSh:1001}. Then we will
remind a general pattern of building forcing notions satisfying those
properties.

\subsection{$A$--reasonably bounding forcing notions}
\begin{definition}
[{\cite[Definition 3.1]{RoSh:860}}]
  \label{p.1A}
  Let $\bbQ$ be a strategically $({<}\kappa)$--complete forcing notion
  and $d\in {}^\kappa(\kappa+1)$ be such that $d(\alpha)$ is a
  cardinal number for all $\alpha<\kappa$.
\begin{enumerate}  
\item For a condition $p\in\bbQ$ we define a game $\Agame(p,\bbQ)$ between
two players, Generic and Antigeneric, as follows. A play of $\Agame(p,\bbQ)$
lasts $\kappa$ steps and during a play a sequence     
\[\Big\langle I_\alpha,\langle p^\alpha_t,q^\alpha_t:t\in I_\alpha\rangle:
\alpha<\kappa\Big\rangle\]
is constructed. At a stage $\alpha<\kappa$ of the game:
\begin{enumerate}
\item[$(\aleph)_\alpha$]  first Generic chooses a non-empty set $I_\alpha$
of cardinality $<d(\alpha)$ and a system $\langle p^\alpha_t:t\in I_\alpha
\rangle$ of conditions from $\bbQ$,
\item[$(\beth)_\alpha$]  then Antigeneric answers by picking a system 
$\langle q^\alpha_t:t\in I_\alpha\rangle$ of conditions from $\bbQ$ such that 
$(\forall t\in I_\alpha)(p^\alpha_t\leq q^\alpha_t)$. 
\end{enumerate}
At the end, Generic wins the play 
\[\Big\langle I_\alpha,\langle p^\alpha_t,q^\alpha_t:t\in I_\alpha\rangle:
\alpha<\kappa\Big\rangle\]  
of $\Agame(p,\bbQ)$ \quad if and only if 
\begin{enumerate}
\item[$(\circledast)^{\rm rc}_{\rm A}$] there is a condition $p^*\in\bbQ$
stronger than $p$ and such that
\[p^*\forces_{\bbQ}\mbox{`` }\big(\forall\alpha<\kappa\big) \big(\exists
  t\in I_\alpha \big)\big(q^\alpha_t\in\name{G}_{\bbQ}\big)\mbox{ ''}.\]
\end{enumerate}
\item We say that a forcing notion $\bbQ$ is {\em reasonably A--bounding 
over $d$\/} if    
\begin{enumerate}
\item[(a)] $\bbQ$ is strategically $({<}\kappa)$--complete,  and 
\item[(b)] for any $p\in\bbQ$, Generic has a winning strategy in the game
$\Agame(p,\bbQ)$. 
\end{enumerate}
If $d(\alpha)=\kappa$ for each $\alpha<\kappa$, then we may omit
$d$ and say {\em reasonably A--bounding\/}.
\end{enumerate}  
\end{definition}

\begin{theorem}
[See {\cite[Theorems 3.1 and 3.2 and Proposition 4.2]{RoSh:860}}] 
  \label{verA}
  Assume that
\begin{enumerate}
\item[(a)] $\kappa$ is a strongly inaccessible cardinal, $d\in {}^\kappa
  (\kappa+1)$ and 
\item[(b)] each $d(\alpha)$ is an infinite regular cardinal satisfying
  \[\big(\forall x\in {}^\alpha d(\alpha)\big)\big(\big|\prod_{\xi<\alpha}
    x(\xi)\big|< d(\alpha)\big),\]
\item[(c)] $\bar{\bbQ}= \langle\bbP_\xi,\name{\bbQ}_\xi:\xi<\gamma\rangle$
  is a $\kappa$--support iteration, and
\item[(d)] for every $\xi<\gamma$,  
\[\forces_{\bbP_\xi}\mbox{`` $\name{\bbQ}_\xi$ is reasonably
    A--bounding over $d$ ''}.\] 
\end{enumerate}
Then $\bbP_{\gamma}=\lim(\bar{\bbQ})$ is $\kappa$--proper and
$\leq^*$--bounding. Moreover
\[\forces_{\bbP_\gamma}\big(\forall x\in \bairek\big)\big(\exists \psi\in
  {\rm Loc}^{\bar{\kappa},d}_\kappa\cap \bV\big)\big(\forall\alpha<\kappa
  \big) \big (x(\alpha)\in \psi(\alpha)\big).\] 
\end{theorem}

\subsection{Pure sequential properness}

\begin{context}
\label{context}
\begin{enumerate}
\item $D$ is the club filter on $\kappa$.
\item A set $\cS\in D^+$ contains all successor ordinals below $\kappa$,
$0\notin\cS$ and $\kappa\setminus\cS$ is unbounded in $\kappa$. For an
ordinal $\gamma<\kappa$ we define  $\cS[\gamma]=\cS\setminus
\{\delta\leq\gamma:\delta\mbox{ is limit }\}$. 
\item $\cR$ is the closure of $\kappa\setminus\cS$.
\item $\diamondsuit_{\cS}$ holds true, so there is a sequence
  $\bar{g}=\langle g_\delta: \delta\in\cS\rangle$ such that
  $g_\delta\in{}^\delta\delta$ and for each $x\in\bairek$ the set
  $\{\delta\in\cS: g_\delta\subseteq x\}$ is stationary.
\item Usualy we will assume $\diamondsuit_\kappa$ is true and
  $\kappa\setminus\cS$ is a non-stationary unbounded set consisting of
  (some) non-successor ordinals only.
\end{enumerate}
\end{context}

\begin{definition}
  [See {\cite[Definition 3.4]{RoSh:1001}}]
\label{n1.0}
Let $\bbQ$ be a forcing notion. A binary relation $R^\pr$ is called {\em a
  $\kappa$--sequential$^+$ purity\/} on $\bbQ$ whenever $\bar{r}\; R^\pr\; r$
implies 
\begin{enumerate}
\item[(a)] $\bar{r}=\langle r_\alpha:\alpha<\delta\rangle$ is a
  $\leq_\bbQ$--increasing sequence of conditions from $\bbQ$ of limit length 
  $\delta<\kappa$, and 
\item[(b)] $r\in\bbQ$ is an upper bound of $\bar{r}$ (i.e.,
  $r_\alpha\leq_\bbQ r$ for all $\alpha<\delta$), and
\item[(c)] if $\bar{r}=\langle r_\alpha:\alpha<\delta\rangle$, $\bar{r}\;
  R^\pr\; s_\beta$ for $\beta<\xi$, $\xi< |\delta|^+$ and $s_\beta\leq
  s_\gamma$ for $\beta<\gamma<\xi$, then there is a condition $s\in\bbQ$
  stronger than all $s_\beta$ (for $\beta<\xi$) and such that $\bar{r}\;
  R^\pr\; s$.
\end{enumerate}
\end{definition}

\begin{definition}
[See {\cite[Definition 3.5]{RoSh:1001}}]
\label{n1.1} 
Let $(\bbQ,\leq)$ be a strategically $({<}\kappa)$--complete forcing notion
and $R^\pr$ be a $\kappa$--sequential$^+$ purity on $\bbQ$. Suppose that a
model $N\prec ({\mathcal H}(\chi),{\in},{<^*_\chi})$ is such that
$|N|=\kappa$, ${}^{<\kappa}N\subseteq N$ and $\kappa,\bbQ,\cS\in N$ and a
function $h:\kappa \longrightarrow N$ is such that its range $\rng(h)$
includes $\bbQ\cap N$. Also, let $\bar{\cI}=\langle\cI_\alpha:\alpha<\kappa
\rangle$ list all open dense subsets of $\bbQ$ belonging to $N$ and let
$\gamma<\kappa$.  
\begin{enumerate}
\item We say that a sequence $\bar{f}=\langle f_\delta:\delta\in\cS\rangle$
  is {\em a $(D,\cS,h)$--semi diamond for $\bbQ$ over $N$\/} if $f_\delta\in
  {}^\delta\delta$ for $\delta\in\cS$ and 
\begin{enumerate}
\item[$(*)$] for every $\leq_\bbQ$--increasing sequence $\bar{p}=\langle 
p_\alpha:\alpha<\kappa\rangle\subseteq \bbQ\cap N$ the set $\{\delta\in\cS:
(\forall\alpha<\delta)(h(f_\delta(\alpha))=p_\alpha)\}$ is stationary.
\end{enumerate}
Below, let $\bar{f}$ be a $(D,\cS,h)$--semi diamond for $\bbQ$ over $N$.  
\item {\em An $(N,h,\bbQ,R^\pr,\bar{f},\bar{\cI})$--candidate\/} is a
  sequence $\bar{q}=\langle q_\delta:\delta\in\cS\mbox{ limit }
  \rangle$ of conditions from $N\cap\bbQ$ satisfying for each limit
  $\delta\in\cS$:   
\begin{enumerate}
\item[(a)] if $h\circ f_\delta=\langle h(f_\delta(\alpha)):
  \alpha<\delta\rangle\subseteq \bbQ\cap N$ and it has an upper bound in
  $\bbQ$, then $h(f_\delta(\alpha))\leq q_\delta$ for all $\alpha<\delta$,
  and  
\item[(b)] if, moreover, $h\circ f_\delta\in \dom(R^\pr)$, then also $h\circ
  f_\delta\; R^\pr\;q_\delta$, and
\item[(c)] if there is $q\in \bigcap\limits_{\alpha<\delta}\cI_\alpha$ such
  that $h\circ f_\delta\; R^\pr\; q$, then also $q_\delta\in
  \bigcap\limits_{\alpha<\delta}\cI_\alpha$. 
\end{enumerate}
\item Let $\bar{q}=\langle q_\delta:\delta\in \cS\ \&\ \delta\mbox{ is limit
    }\rangle$ be a candidate and $r\in\bbQ$. We define a game
    $\Game^\cS_\gamma(r,N,h,\bbQ,R^\pr,\bar{f},\bar{q})$ of two players,
    {\em Generic\/} and {\em Antigeneric}, as follows. A play lasts 
    $\leq\kappa$ moves and in the $i^{\rm th}$ move the players try to
    choose conditions $r_i^-,r_i \in \bbQ$ and a closed unbounded set
    $C_i\subseteq \kappa$ so that\footnote{Remember Context \ref{context}}:  

\begin{enumerate}
\item[(a)] $r\leq r_i$, and $r_i^-\in N$, and if $i\notin\cS[\gamma]\cap \cR$
  then  $r_i^-\leq r_i$,  
\item[(b)] $(\forall i<j<\kappa)(r_i\leq r_j\ \&\ r_i^-\leq r_j^-)$, and  
\item[(c)] Generic chooses $r_i^-,r_i,C_i$ if $i\in\cS[\gamma]$, and
  Antigeneric chooses $r_i^-,r_i,C_i$ if $i\notin\cS[\gamma]$.    
\end{enumerate}
At the end Generic wins the play whenever both players always had legal
moves (so the game lasted $\kappa$ steps) and 
\begin{enumerate}
\item[$(\circledast)$] if $\delta\in\cS[\gamma]\cap \bigcap\limits_{i<\delta}
  C_i$ is a limit ordinal and $h\circ f_\delta$ is an increasing sequence of 
  conditions in $\bbQ$ such that for all $\alpha<\delta$ we have
  $h(f_\delta(\alpha+1))= r^-_{\alpha+1}$, then $q_\delta\leq r_\delta$ and
  $h\circ f_\delta\; R^\pr\; r_\delta$. 
\end{enumerate}
\item Let $\bar{q}$ be a candidate (for $N,h,\bbQ,R^\pr,\bar{f}$ and
  $\bar{\cI}$). A  condition $r\in\bbQ$ is {\em generic for the candidate
    $\bar{q}$ over $N,h,\bbQ,R^\pr,\bar{f},\cS,\gamma$\/} if Generic has a
  winning strategy in the game $\Game^\cS_\gamma(r,N,h,\bbQ,R^\pr,\bar{f}, 
  \bar{q})$. 
\end{enumerate}
\end{definition}

\begin{observation}
  If $\cS$ is as in Context \ref{context} and $h,N$ are as in Definition
  \ref{n1.1}, then there exists a $(D,\cS,h)$--semi diamond for $\bbQ$ over
  $N$. 
\end{observation}

\begin{definition}
[See {\cite[Definition 3.6]{RoSh:1001}}]
\label{Def3.6} 
Let $\bbQ$ be a strategically $({<}\kappa)$--complete forcing notion. We
say that $\bbQ$ is {\em purely sequentially$^+$ proper over $(D,\cS)$--semi
  diamonds\/} whenever the following condition $(\odot)$ 
  is satisfied.  
\begin{enumerate}
\item[$(\odot)$] Assume that $\chi$ is a large enough regular cardinal and 
  $N\prec (\cH(\chi),\in,<^*_\chi)$, $|N|=\kappa$, ${}^{<\kappa} N\subseteq
N$ and $\kappa,\bbQ,\cS\in N$. {\em Then\/} there exists a
$\kappa$--sequential$^+$ purity $R^\pr$ on $\bbQ$ such that for every
ordinal $\gamma<\kappa$, a condition $p\in \bbQ\cap N$ and every
$\bar{\cI},h,\bar{f},\bar{q}$ satisfying    
\begin{itemize}
\item $\bar{\cI}=\langle \cI_\alpha:\alpha<\kappa\rangle$ lists all open
  dense subsets of $\bbQ$ from $N$, 
\item a function $h:\kappa\longrightarrow N$ is such that $\bbQ\cap
  N\subseteq\rng(h)$, and 
\item a sequence $\bar{f}$ is a $(D,\cS,h)$--semi diamond for $\bbQ$, and 
\item $\bar{q}$ is an  $(N,h,\bbQ,R^\pr,\bar{f},\bar{\cI})$--candidate,
\end{itemize}
there is a condition $r\geq p$ generic for $\bar{q}$ over $N,h,\bbQ,
R^\pr,\bar{f},\cS,\gamma$.   
\end{enumerate}
\end{definition}

\begin{theorem}
  [See {\cite[Theorem 4.1 and Proposition 3.8(6)]{RoSh:1001}}]
\label{n3.7} 
Assume Context \ref{context}. Let $\bar{\bbQ}=\langle\bbP_\alpha,
\name{\bbQ}_\alpha:\alpha<\zeta^*\rangle$ be a $\kappa$--support iteration 
such that for each $\alpha<\zeta^*$   
\[\forces_{\bbP_\alpha}\mbox{`` $\name{\bbQ}_\alpha$ is purely
  sequentially$^+$ proper over $(D,\cS)$--semi diamonds ''.}\] 
Then $\bbP_{\zeta^*}=\lim(\bar{\bbQ})$ is $\kappa$--proper in the standard
sense.
\end{theorem}

\subsection{The forcing notions}
The forcing notions considered in this note fit the general pattern
described in Definition \ref{boundedPQ} below. The required properties of
these forcing notions actually follow from results in \cite{RoSh:777,
  RoSh:860, RoSh:1001}. However, some of them were not explicitly stated or
proven; in other cases, the assumptions were stronger than what is actually
needed. Therefore, for the sake of clarity and completeness, we will give
some of the arguments here.
\medskip

\begin{definition}
  [See {\cite[Definition 3.10]{RoSh:942}}, {\cite[Definition
      2.20(1)]{RoSh:1001}}] 
\label{boundedPQ}
Assume that 
\begin{itemize}
\item $\kappa$ is strongly inaccessible, $f$ is a function with domain
  $\kappa$ and $|f(\alpha)|\leq \kappa$ (for $\alpha<\kappa$), 
\item $\bar{F}=\langle F_t:t\in\bigcup\limits_{\alpha<\kappa}
  \prod\limits_{\xi<\alpha} f(\xi)\rangle$ where $F_t$ is a
  ${\leq}|\alpha|$--complete filter on $f(\alpha)$ whenever $t\in
  \prod\limits_{\xi<\alpha}f(\xi)$, $\alpha<\kappa$. 
\end{itemize}
\begin{enumerate}
\item We define a forcing notion $\bbQ^2_{f,\bar{F}}$ as follows.\\ 
{\bf A condition} in $\bbQ^2_{f,\bar{F}}$ is a complete
  $\kappa$--tree $T\subseteq \bigcup\limits_{\alpha<\kappa}
  \prod\limits_{\xi<\alpha}f(\xi)$ such that  
\begin{enumerate}
\item[(a)] for every $t\in T$, either $|\suc_T(t)|=1$ or $\suc_T(t)\in  
  F_t$, and
\item[(b)] $(\forall t\in T)(\exists s\in T)(t\vtl s\ \&\ |\suc_T(s)|>1)$,
  and 
\item[(c)$^2$] if $\langle t_i:i<j\rangle\subseteq T$ is $\vtl$--increasing,
  $j<\kappa$, $|\suc_T(t_i)|>1$ for all $i<j$ and
  $t=\bigcup\limits_{i<j}t_i$, then ($t\in T$ and) $|\suc_T(t)|>1$.
\end{enumerate}
\noindent {\bf The order} of $\bbQ^2_{f,\bar{F}}$ is the reverse inclusion. 
\item Forcing notions $\bbQ^\ell_{f,\bar{F}}$ for $\ell=3,4$ are defined
  similarly, but the demand (c)$^2$ is replaced by the respective
  (c)$^\ell$:  
\begin{enumerate}
\item[(c)$^3$] for some club $C$ of $\kappa$ consisting of limit ordinals
  only we have 
\[\big(\forall t\in T\big)\big(\lh(t)\in C\ \Leftrightarrow\ \suc_T(T)\in
F_t \big).\]
\item[(c)$^4$] $(\forall t\in T)(\mrot(T)\vtl t\ \Rightarrow\
  |\suc_T(t)|>1)$. 
\end{enumerate}
\end{enumerate}
\end{definition}

\begin{remark}
  In Definition \ref{boundedPQ} we do not demand that the filters $F_t$ are
  non-principal. In Section 5 we will be considering
  $F_t=\big\{f(\lh(t))\big\}$. 
\end{remark}

\begin{lemma}
  \label{q2isAb}
Assume that $\kappa,f$ and $\bar{F}$ are as in Definition \ref{boundedPQ}
and $\ell\in \{2,3,4\}$. 
\begin{enumerate}
\item The forcing notions $\bbQ^\ell_{f,\bar{F}}$ (for $\ell=2,3,4$) are
  strategically $({<}\kappa)$--complete.
\item If $\delta<\kappa$, $\bar{S}=\langle S_i:i<\delta\rangle\subseteq
  \bbQ^\ell_{f,\bar{F}}$ is $\leq_{\bbQ^\ell_{f,\bar{F}}}$--increasing,
    $\mrot(S_i)=\mrot(S_0)$ for all $i<\delta$ and $\delta\leq
    \lh(\mrot(S_0))$, then $T\stackrel{\rm def}{=} \bigcap\limits_{i<\delta}
    S_i\in \bbQ^\ell_{f,\bar{F}}$ is a $\leq_{\bbQ^\ell_{f,\bar{F}}}$--upper
    bound of $\bar{S}$ and $\mrot(T)=\mrot(S_0)$. 
\item Suppose $\gamma\leq \kappa$ is a limit ordinal and $S_\delta\in
  \bbQ^\ell_{f,\bar{F}}$ are such that for $\delta<\gamma$:
  \begin{enumerate}
  \item[(i)]  $S_{\delta+1}\subseteq S_\delta$ and $S_\delta\cap
    \prod\limits_{\xi<\delta} f(\xi)=S_{\delta+1}\cap\prod\limits_{\xi<
      \delta} f(\xi)$, 
  \item[(ii)] if $\delta$ is limit, then
    $S_\delta=\bigcap\limits_{i<\delta} S_i$,
  \item[(iii)] if $t\in \spl(S_\delta)\cap \prod\limits_{\xi<\delta}
    f(\xi)$, then $t\in\spl(S_{\delta+1})$.
  \end{enumerate}
 Then $S_\gamma\stackrel{\rm def}{=} \bigcap\limits_{\delta<\gamma}
 S_\delta\in  \bbQ^\ell_{f,\bar{F}}$.
\item If $|f(\alpha)|<\kappa$ for all $\alpha$, then
  $\bbQ^3_{f,\bar{F}}$ is a dense subset of $\bbQ^2_{f,\bar{F}}$.

\end{enumerate}
\end{lemma}

\begin{proof}
(1) \quad We will describe a strategy of Complete in the game
$\Game^\kappa_0\big(\bbQ^\ell_{f,\bar{F}}\big)$. Suppose we are at a stage
$\alpha<\kappa$ of the play and a sequence $\big\langle
(S_i,T_i):i<\alpha\big\rangle$ has been already constructed. Suppose also
that we know it has an $\leq_{\bbQ^\ell_{f,\bar{F}}}$--upper bound (relevant
if $\alpha$ is limit). Now Incomplete puts forward a condition $S_\alpha \in
\bbQ^\ell_{f,\bar{F}}$ stronger than all $T_i$ for $i<\alpha$. The strategy
of Complete is to choose (the $<^*_\chi$--first) $t\in S_\alpha$ such that
$\lh(t)=\lh\big(\mrot(S_\alpha)\big)+\omega$ and after this she plays
$T_\alpha=\big(S_\alpha\big)_t$. 
\medskip

In order to show that the strategy described above is a winning strategy for
Complete, it is enough to prove that
\begin{enumerate}
\item[$(\oplus)$] if $\alpha<\kappa$ is a limit ordinal and $\langle
  T_i:i<\alpha \rangle \subseteq \bbQ^\ell_{f,\bar{F}}$ is
  $\leq_{\bbQ^\ell_{f,\bar{F}}}$--increasing sequence of conditions such
  that $\lh\big(\mrot(T_i)\big)\geq i$, then $\bigcap\limits_{i<\alpha}
  T_i\in  \bbQ^\ell_{f,\bar{F}}$. 
\end{enumerate}
So let $\langle T_i:i<\alpha \rangle \subseteq \bbQ^\ell_{f,\bar{F}}$ be as
in the assumptions of $(\oplus)$ and let $T=\bigcap\limits_{i<\alpha}
T_i$. Clearly $T$ is a tree. If $t\in\spl(T_i)$ for each $i<\alpha$, then
$t\in T$, $\lh(t)\geq \alpha$ and $\suc_T(t)=\bigcap\limits_{i<\alpha}
\suc_{T_i}(t)\in F_t$ (as $F_t$ is ${\leq}|\alpha|$--complete). Hence
$\bigcap\limits_{i<\alpha} \spl(T_i)=\spl(T)$ and for each $t\in T$, either
$|\suc_T(t)|=1$ or $\suc_T(t)\in F_t$ (remember condition
\ref{boundedPQ}(1)(a)). In particular, at this moment we know that $T$ is a  
complete $\kappa$--tree.

Let $s^*=\bigcup\limits_{i<\alpha} \mrot(T_i)$.

We will now consider the cases $\ell=2$ and $\ell=4$ separately. The case of
$\ell=3$ is similar.
\medskip

\noindent {\sc Case $\ell=4$}\\
If $s^*\trianglelefteq s\in T$, then $s\in \bigcap\limits_{i<\alpha}
\spl(T_i) =\spl(T)$. Thus $s^*=\mrot(T)$ and $T\in \bbQ^4_{f,\bar{F}}$. 
\medskip

\noindent {\sc Case $\ell=2$}\\
If $\mrot(T_j)=s^*$ for some $j<\alpha$, then easily $s^*\in
\bigcap\limits_{i<\alpha} \spl(T_i)$. Otherwise there is an increasing
sequence $\langle i_\gamma:\gamma<\cf(\alpha)\rangle$ cofinal in $\alpha$
and such that $\mrot(T_{i_{\gamma_0}})\vtl \mrot(T_{i_{\gamma_1}})$ whenever
$\gamma_0<\gamma_1<\cf(\alpha)$. Since $\mrot(T_{i_\gamma})\in \spl(T_j)$
for all $j\leq i_\gamma$, we may use property \ref{boundedPQ}(1)(c)$^2$ to
argue that $s^*\in \bigcap\limits_{i<\alpha} \spl(T_i)$. In either case we
conclude $s^*= \mrot(T)$. 
\medskip

To verify \ref{boundedPQ}(1)(b) for $T$ assume $t\in T$. We attempt to
construct inductively a sequence $\langle t_i:i<\alpha\rangle\subseteq T$
such that $t=t_0$, $t_i\vtl t_j$ for $i<j<\alpha$, and $t_{i+1}\in
\spl(T_{i+1})$ for all $i<\alpha$. Suppose $t_i\in T$ has been
determined. Pick any $s\in T$ properly extending $t_i$. If $s\in\spl(T)$,
then we get the required property and we interrupt our
construction. Otherwise, take $t_{i+1}$ to be the shortest sequence in
$\spl(T_{i+1})$ extending $s$ (possibly $t_{i+1}=s$). Then necessarily
$t_{i+1} \in T$.  If $i<\alpha$ is limit and the construction has not
terminated by this point, we let $t_i=\bigcup\limits_{j<i} t_j$. We know
that $t_i\in T\subseteq T_i$.

If we did manage to construct $\langle t_i:i<\alpha\rangle$, then by the
same argument as for the roots, $\bigcup\limits_{i<\alpha} t_i\in
\bigcap\limits_{i<\alpha} \spl(T_i)=\spl(T)$.
\medskip

Finally, demand \ref{boundedPQ}(1)(c)$^2$ for $T$ is an immediate
consequence of the same demand for the $T_i$s and the equality
$\bigcap\limits_{i<\alpha} \spl(T_i)=\spl(T)$.
\medskip

\noindent  (2) \quad Included in the proof of (1).
\medskip

\noindent  (3) \quad Let us consider the case $\gamma=\kappa$ only (if
$\gamma<\kappa$ then the arguments are similar).

Easily, $S_\kappa\subseteq \bigcup\limits_{\delta<\kappa}
\prod\limits_{\xi<\delta} f(\xi)$ is a tree.

Suppose $t\in S_\kappa$ is of length $\lh(t)=\delta$. By assumption (3)(i),
$\suc_{S_{\delta+1}}(t)=\suc_{S_i}(t)$ for all $\delta< i<\kappa$, so
$\suc_{S_\kappa} (t)=\suc_{S_{\delta+1}}(t)$ is either a singleton or a set
from $F_t$. Now it should be also clear that $S_\kappa$ is a complete
$\kappa$--tree.

Let us argue that $S_\kappa$ satisfies condition
\ref{boundedPQ}(1)(b). Suppose $t\in S_\kappa$, $\delta_0=\lh(t)$. We
attempt to choose inductively $\vtl$--increasing sequence $\langle
t_n:n<\omega\rangle \subseteq S_\kappa$ so that
\begin{itemize}
\item $t_0=t$, $t_n\vtl t_{n+1}$ and
\item $t_{n+1}\in\spl (S_{\delta_n})$ where $\delta_n=\lh(t_n)$. 
\end{itemize}
Suppose we have chosen $t_n$ of length $\lh(t_n)=\delta_n$. If $t_n\in 
\spl(S_{\delta_n})$, then the procedure stops and we note that (3)(iii)
implies $t_n\in\spl(S_{\delta_n+1})$ and (3)(i) gives us then
$\suc_{S_\kappa}(t_n)=\suc_{S_{\delta_n+1}}(t_n)\in F_{t_n}$. If $t_n\notin 
\spl(S_{\delta_n})$, we pick the shortest $t_{n+1}\in \spl(S_{\delta_n})$
extending $t_n$; necessarily $t_{n+1}\in S_\kappa$.

If the above process did not stop and we constructed $\langle t_n:n
<\omega\rangle\subseteq S_\kappa$ with $t_n\vtl
t_{n+1}\in\spl(S_{\delta_n})$, where $\delta_n=\lh(t_n)$, then we put
$t^*=\bigcup\limits_{n<\omega} t_n$. Let
$\delta^*=\lh(t^*)=\sup\limits_{n<\omega} \delta_n$. Since
$\spl(S_{\delta_m})\subseteq \spl(S_{\delta_n})$ for $m\geq n$, we have
$t_m\in \spl(S_{\delta_n})$ for $m>n$. Then, regardless of the value of
$\ell$, we will have $t^*\in \spl(S_{\delta_n})$ for each $n<\omega$ and
\[\suc_{S_{\delta^*}}(t^*)=\bigcap_{n<\omega} \suc_{\delta_n}(t^*) \in
  F_{t^*}\]
(remember (3)(ii)). Now (3)(iii) implies $t^*\in \spl(S_{\delta^*+1})$ and
(3)(i) gives us then $\suc_{S_\kappa}(t^*)=\suc_{S_{\delta^*+1}}(t^*)\in
F_{t^*}$. Thus $t\vtl t^*\in\spl(S_\kappa)$.
\medskip

Finally, concerning demands (c)$^\ell$ of \ref{boundedPQ}, we consider the
three values of $\ell$ separately. 
\medskip

\noindent {\sc Case $\ell=2$}\\
Assume $t_j\in\spl(S_\kappa)$, $t_i\vtl t_j$ for
  $i<j<j^*<\kappa$. Let $t=\bigcup\limits_{j<j^*} t_j$,
  $\delta=\lh(t)+5<\kappa$. Then $t_j\in\spl(S_\delta)$ (for $j<j^*$), so
  $t\in \spl(S_\delta)$ (by (c)$^2$ for $S_\delta$). However, $S_\kappa \cap
  \prod\limits_{\xi<\delta} f(\xi)=S_\delta \cap \prod\limits_{\xi<\delta}
  f(\xi)$, so $t\in \spl(S_\kappa)$.
\medskip

\noindent {\sc Case $\ell=3$}\\  
Let $C_\delta\subseteq \kappa$ be the club set witnessing
  (c)$^3$ for $S_\delta$ (for $\delta<\kappa$). Let
  $C=\mathop{\triangle}\limits_{\delta<\kappa} C_\delta\setminus
  \{0\}$.

  Suppose $t\in S_\kappa$, $\delta^*=\lh(t)\in C$. Then $\delta^*$ is
  limit and  $t\in\spl(S_\delta)$ for all $\delta<\delta^*$. By (3)(ii) we
  know $\suc_{S_{\delta^*}}(t)=\bigcap\limits_{\delta<\delta^*}
  \suc_{S_\delta}(t) \in F_t$ and by (3)(iii, i) we conclude
  $t\in\spl(S_\kappa)$.

  Conversely, if $t\in\spl(S_\kappa)$, then $t\in\spl(S_\delta)$ for each
  $\delta<\kappa$. Hence $\lh(t)\in C$.
\medskip

\noindent {\sc Case $\ell=4$}\\
Let $s_\delta=\mrot(S_\delta)$ for $\delta<\kappa$. We claim that the
sequence $\langle s_\delta:\delta<\kappa\rangle$ is eventually constant. If
not, we could find a limit ordinal $\delta^*<\kappa$ such that
$\sup\limits_{\delta<\delta^*} \lh(s_\delta)=\delta^*$. Let
$s^*=\bigcup\limits_{\delta<\delta^*} s_\delta$. Then $s^*\in
\bigcap\limits_{\delta<\delta^*} \spl(S_\delta)$, so in a standard by now
way we get $s^*\in\spl(S_{\delta^*+1})$ and $s^*\in\spl(S_\kappa) \subseteq
\bigcap\limits_{\delta<\kappa} \spl(S_\delta)$. Therefore
$s_\delta\trianglelefteq s^*$ for all $\delta<\kappa$, a contradiction.

Thus for some $\delta_0<\kappa$, for all $\delta\geq \delta_0$, we have
$s_\delta=s_{\delta_0}$. Now, if $s_{\delta_0}\trianglelefteq t\in
S_\kappa$, then $t\in \spl(S_\delta)$ for all $\delta$ and hence $t\in
\spl(S_\kappa)$.  
\end{proof}

\begin{proposition}
  \label{getAbound}
If $d\in {}^\kappa (\kappa+1)$, $\big|\prod\limits_{\xi\leq\alpha}
  f(\xi)\big|<d(\alpha)$ for all $\alpha$, then $\bbQ^2_{f,\bar{F}}$ is
  reasonably A--bounding over $d$.   
\end{proposition}

\begin{proof}
  Assume $d:\kappa\longrightarrow (\kappa+1)$ is such that
$\big|\prod\limits_{\xi\leq \alpha}f(\xi)\big|<d(\alpha)$ for each
$\alpha<\kappa$.

Let $R\in \bbQ^2_{f,\bar{F}}$. We will describe a winning strategy for
Generic in $\Agame(R, \bbQ^2_{f,\bar{F}})$.

During the course of the play, in addition to her innings $I_\alpha$,
$\bar{S}^\alpha= \langle S^\alpha_t:t\in I_\alpha\rangle$,  Generic will be
choosing aside conditions $R_\alpha\in \bbQ^2_{f,\bar{F}}$. These objects  
are to be chosen so that if $I_\alpha$, $\bar{S}^\alpha=\langle
S^\alpha_t:t\in I_\alpha\rangle$, $\bar{T}^\alpha=\langle T^\alpha_t:t\in
I_\alpha\rangle$ are what the players put at stage $\alpha$, and
$R_\alpha$ are written down aside by Generic, then the following
demands $(\otimes)_0$--$(\otimes)_3$ are satisfied.
\begin{enumerate}
\item[$(\otimes)_0$]  $I_\alpha$, $\bar{S}^\alpha$ and $\bar{T}^\alpha$
  abide by the rules of the game $\Agame(R, \bbQ^2_{f,\bar{F}})$. 
\item[$(\otimes)_1$] $R_\alpha\cap\prod\limits_{\xi\leq \alpha} f(\xi)=
  R_{\alpha+1} \cap \prod\limits_{\xi\leq \alpha} f(\xi) =I_\alpha$,
  $R_\alpha \leq_{\bbQ^2_{f,\bar{F}}} R_{\alpha+1}$ and
  $R=R_0$.
\item[$(\otimes)_2$] If $\gamma$ is limit, then
  $R_\gamma=\bigcap\limits_{\alpha<\gamma} R_\alpha$.
\item[$(\otimes)_3$] If $t\in I_\alpha$ then $S^\alpha_t=
  \big(R_\alpha\big)_t$ and $\big(R_{\alpha+1}\big)_t=T^\alpha_t$. 
\end{enumerate}
Conditions $(\otimes)_0$--$(\otimes)_3$ determine a strategy for
Generic. If $\big\langle I_\alpha,\bar{S}^\alpha,\bar{T}^\alpha:\alpha <\kappa
\big\rangle$ is a result of a play agreeing with the strategy described
above and $\langle R_\alpha:\alpha<\kappa\rangle$ is the sequence
of the objects chosen aside, then Generic puts $R^*=
\bigcap\limits_{\alpha<\kappa} R_\alpha$. By Lemma \ref{q2isAb}(3),
$R^*\in \bbQ^2_{f,\bar{F}}$  and clearly 
\[R^*\forces_{\bbQ^2_{f,\bar{F}}} (\forall \alpha<\kappa)(\exists t\in
  I_\alpha) (T^\alpha_t\in \name{G}_{\bbQ^2_{f,\bar{F}}}).\] 
\end{proof}

\begin{proposition}
  \label{getpurity}
  Assume $\kappa,f$ and $\bar{F}$ are as in Definition \ref{boundedPQ} and
  $\cS$ is as in Context \ref{context}. Let $\ell\in \{2,3,4\}$. Then the
  forcing notion $\bbQ^\ell_{f,\bar{F}}$ is purely sequentially$^+$ proper
  over $(D,\cS)$--semi diamonds. 
\end{proposition}

\begin{proof}
  First we define a binary relation $R^\pr$ by

  $\bar{T}\; R^\pr\; T$\quad if and only if  
\begin{itemize}
\item $\bar{T}=\langle T_\alpha:\alpha<\delta\rangle$ is a
  $\leq_{\bbQ^\ell_{f,\bar{F}}}$--increasing sequence of conditions from
  $\bbQ^\ell_{f,\bar{F}}$ such that $\delta<\kappa$ is a limit ordinal and
  for $\alpha<\delta$ we have $\alpha\leq
  \lh\big(\mrot(T_\alpha)\big)<\delta$, and  
\item $T\in \bbQ^\ell_{f,\bar{F}}$ is a
  $\leq_{\bbQ^\ell_{f,\bar{F}}}$--upper bound of the sequence $\bar{T}$ with
  $\mrot(T)=\mrot(\bigcap\limits_{\alpha<\delta} T_\alpha)$.  By $(\oplus)$
  from the proof of Lemma \ref{q2isAb}(1), $\bigcap\limits_{\alpha<\delta}
  T_\alpha\in \bbQ^\ell_{f,\bar{F}}$. We also know that
  $\bigcup\limits_{\alpha<\delta} \mrot(T_\alpha)= 
  \mrot(\bigcap\limits_{\alpha<\delta} T_\alpha)$.
\end{itemize}
By Lemma \ref{q2isAb}(2), $R^\pr$ is a $\kappa$--sequential$^+$ purity on
$\bbQ^\ell_{f,\bar{F}}$. Suppose that $N\prec (\cH(\chi),\in,<^*_\chi)$ and
$h,\bar{f},\bar{\cI},\gamma$ are as in the assumptions of Definition
\ref{Def3.6}$(\odot)$. Let $\bar{S}=\langle S_\delta: \delta\in\cS \mbox{ is
  limit}\;\rangle$ be an $(N,h, \bbQ^\ell_{f,\bar{F}},R^\pr, \bar{f},
\bar{\cI})$--candidate and let $T\in \bbQ^\ell_{f,\bar{F}}\cap N$. We will
construct a condition $T^*\geq T$ generic for $\bar{S}$ over $N,h,
\bbQ^\ell_{f,\bar{F}}, R^\pr, \bar{f},\cS, \gamma$. To this end by induction
on $\alpha< \kappa$ we choose a $\leq_{\bbQ^\ell_{f,\bar{F}} }$--increasing 
sequence $\langle T_\alpha:\alpha<\kappa\rangle\subseteq
\bbQ^\ell_{f,\bar{F}} \cap N$ as follows. 
\begin{enumerate}
\item[$(\triangle)_1$] $T_0=T$ and if $\delta$ is a limit ordinal,
  then $T_\delta=\bigcap\limits_{\xi<\delta}T_\xi$.  
\item[$(\triangle)_2$] Suppose that $\alpha=\delta+1$,
  $\delta\in\cS[\gamma]$ is a limit  ordinal, and $h\circ f_\delta$ is a
  $\leq_{\bbQ^\ell_{f,\bar{F}}}$--increasing sequence of conditions from
  $\bbQ^\ell_{f,\bar{F}}\cap N$ such that for all successor $\beta<\delta$
  we have  
\[\beta\leq\lh\big(\mrot(h(f_\delta(\beta)))\big)<\delta\ \mbox{ and }\
T_\beta\leq h(f_\delta(\beta)).\]   
Let $s=\mrot\big(\bigcap\limits_{\beta<\delta} h(f_\delta(\beta))\big)
=\bigcup\limits_{\beta<\delta} \mrot\big(h(f_\delta(\beta))\big)$. Then
$s\in\spl(T_\delta)$ and $s=\mrot(S_\delta)$, and $T_\delta\leq S_\delta$.
Let $T_\alpha^*=\big\{t\in T_\delta: s \trianglelefteq t\ \Rightarrow\ t\in
S_\delta\big\}$.

If $\ell\neq 3$ then $T^*_\alpha\in \bbQ^\ell_{f,\bar{F}}$ and we put
$T_\alpha= T^*_\alpha$. Otherwise, we choose $T_\alpha\in
\bbQ^\ell_{f,\bar{F}}$ such that $T_\alpha\subseteq T^*_\alpha$ and
$T_\alpha\cap \prod\limits_{\xi<\alpha} f(\xi)=  
T^*_\alpha\cap \prod\limits_{\xi<\alpha} f(\xi)$. (Possible as we just have
to ``trim'' $T^*_\alpha$ to a subtree with splittings on $(C_1\cap C_2)\cup
(C_1\cap\alpha)$, where $C_1$ is given by (c)$^3$ for $T_\delta$ and $C_2$
is given by (c)$^3$ for $S_\delta$. Note that $\delta\in C_1$.)     
\item[$(\triangle)_3$] If $\alpha=\delta+1$ and the clause
$(\triangle)_2$ does not apply, then $T_\alpha=T_\delta$.
\end{enumerate}
Plainly, the sequence $\langle T_\alpha:\alpha<\kappa\rangle$ satisfies the
assumptions of Lemma \ref{q2isAb}(3). At limit stages $\delta<\kappa$ this
implies that  $T_\delta= \bigcap\limits_{\beta <\delta} T_\beta\in
\bbQ^\ell_{f,\bar{F}}$ (so the construction can be carried out) and for
$\kappa$ we get $T^*\stackrel{\rm def}{=} \bigcap\limits_{\alpha<\kappa}
T_\alpha\in \bbQ^\ell_{f,\bar{F}}$. 

\begin{claim}
\label{cl10}
$T^*$ is $(N, \bbQ^\ell_{f,\bar{F}})$--generic in the standard sense.  
\end{claim}

\begin{proof}[Proof of the Claim]
Suppose towards contradiction that there are $\xi_0<\kappa$ and a condition  
$T^+\geq T^*$ such that $T^+$ is incompatible with every member of
$\cI_{\xi_0}\cap N$. Construct inductively a sequence $\langle U_\alpha:
\alpha<\kappa\rangle$ of conditions in $\bbQ^\ell_{f,\bar{F}}$ such that  
\begin{itemize}
\item $T^+\leq U_\beta\leq U_\alpha$, $\lh(\mrot(U_\beta))< \lh(\mrot(
  U_\alpha))$ for $\beta<\alpha<\kappa$, and 
\item $U_{\alpha+1}\in\bigcap\limits_{\xi\leq\alpha} \cI_\xi$ and if
  $\alpha$ is limit, then $U_\alpha=\bigcap\limits_{\beta<\alpha} U_\beta$
  (for all $\alpha<\kappa$).  
\end{itemize}
Let $u_\alpha=\mrot(U_\alpha)$. Since $U_\alpha\subseteq T^+\subseteq
T^*\subseteq T_\alpha$ we see that $u_\alpha\in T_\alpha$ and
$(T_\alpha)_{u_\alpha} \leq U_\alpha$. Now take a limit ordinal
$\delta\in\cS[\gamma]$ such that  $\xi_0<\delta$ and   
\[h\circ f_\delta=\langle (T_\alpha)_{u_\alpha}:\alpha<\delta\rangle,\quad 
u_\delta=\bigcup_{\alpha<\delta} u_\alpha,\quad \lh(u_\delta)= 
\delta\ \mbox{ and }\ U_\delta\in \bigcap_{\alpha<\delta}\cI_\alpha.\]
Then $h\circ f_\delta\; R^\pr\; U_\delta$, so by \ref{n1.1}(2)(c) we
conclude $S_\delta\in \bigcap\limits_{ \alpha<\delta}\cI_\alpha$ (in
particular $S_\delta\in\cI_{\xi_0}\cap N$). But $S_\delta\leq
(T_{\delta+1})_{u_\delta} \leq (T^*)_{u_\delta}\leq (T^+)_{u_\delta}$,
contradicting the choice of $T^+$.   
\end{proof}

\begin{claim}
  \label{cl11}
$T^*$ is generic for $\bar{q}$ over $N,h, \bbQ^\ell_{f,\bar{F}},
R^\pr,\bar{f},\cS,\gamma$.   
\end{claim}

\begin{proof}[Proof of the Claim]
Consider the following strategy $\st$ of  Generic in the game
$\Game^\cS_\gamma(T^*,N,h,\bbP,R^\pr,\bar{f}, 
\bar{q})$. At a stage $i\in \cS[\gamma]$, when a sequence $\langle
r_j^-,r_j, C_j:j<i\rangle$ has been already constructed, Generic picks
$r^-_i,r_i,C_i$ so that  
\begin{enumerate}
\item[$(\triangle)_4$] if $i$ is limit, then $r^-_i=\bigcap\limits_{j<i}
r^-_j$, $r_i=\bigcap\limits_{j<i} r_j$ and $C_i=\bigcap\limits_{j<i} C_j$, 
\item[$(\triangle)_5$] if $i$ is a successor, say $i=j+1$, then $r^-_i,r_i$  
are chosen so that $r_j\leq r_i$, $\mrot(r_j)\vtl\mrot(r_i^-)=\mrot(r_i)$,
$r^-_j\leq r^-_i$, $T_i\leq r^-_i$ and $r^-_i\leq r_i$. Also, $C_i=
[\lh(\mrot(r_i))+\omega,\kappa)$.
\end{enumerate}
These choices can be done as follows. By $(\triangle)_5$ at previous stages,
\begin{itemize}
\item for 
each limit $i<\kappa$ the intersection of conditions played so far is in 
$\bbQ^\ell_{f,\bar{F}}$,
\item $r^-_i\leq r_i$ for limit $i\in \cS[\gamma]$ (and hence this
  inequality will hold for all $i$).
\end{itemize}
At a stage $i=j+1$, since $r^-_j\leq r_j$, $T_i\leq T^*\leq r_0\leq 
r_j$ and $T^*$ is $(N,\bbP)$--generic, we may find $r^*\in
\bbQ^\ell_{f,\bar{F}} \cap N$ stronger than both $r^-_j$ and $T_i$ and
compatible with $r_j$. Take $r_i^*$ stronger than both $r^*$
and $r_j$ and then pick $t\in \spl(r_i^*)$ such that
$\lh(t)>\lh(\mrot(r_j))$. Set $r^-_i= (r^*)_t$ and $r_i=(r_i^*)_t$. 

There are always legal moves for Antigeneric (for instance described by the
rules in $(\triangle)_4$).  

Now suppose that $\langle r^-_i,r_i,C_i:i<\kappa\rangle$ is a play of
$\Game^\cS_\gamma(T^*,N,h, \bbQ^\ell_{f,\bar{F}} ,R^\pr,\bar{f},\bar{q})$ in
which Generic  follows the strategy $\st$ described above. Let
$\delta\in\cS[\gamma] \cap \bigcap\limits_{i<\delta} C_i$ be a limit ordinal
such that $h\circ f_\delta$ is an increasing sequence of conditions in
$\bbQ^\ell_{f,\bar{F}}$ satisfying  $h(f_\delta(i))=r_i^-$  for all
successor $i<\delta$. Then our choices in $(\triangle)_5+(\triangle)_4$
imply that $i \leq \lh\big(\mrot(h(f_\delta(i)))\big)<\delta$ and
$h(f_\delta(i)) \geq T_i$ (for all successor $i<\delta$). Let
$t=\bigcup\limits_{i<\delta} \mrot(r^-_i)$. Then $t=\mrot\big(
\bigcap\limits_{i<\delta} r_i\big)= \mrot \big(\bigcap\limits_{i<\delta}
r^-_i\big)$ and using $(\triangle)_2+(\triangle)_4$ we conclude
$S_\delta\leq (T_{\delta+1})_t \leq (T^*)_t\leq (r_\delta)_t=r_\delta$ and
$h\circ f_\delta\; R^\pr\; r_\delta$.     
\end{proof}
\end{proof}

\begin{observation}
  \label{genseq}
  Assume $\kappa,f$ and $\bar{F}$ are as in Definition
  \ref{boundedPQ}. Let $\name{W}=\name{W}_{\bbQ^\ell_{f,\bar{F}}}$ be
  the canonical name for the $\kappa$--sequence added by
  $\bbQ^\ell_{f,\bar{F}}$ (i.e., $T\forces_{\bbQ^\ell_{f,\bar{F}}}
  \mrot(T) \vtl \name{W}$). Let $T\in \bbQ^\ell_{f,\bar{F}}$.
  \begin{enumerate}
  \item If $\ell=2$ then $T\forces_{\bbQ^2_{f,\bar{F}}}$``
    $\{\alpha<\kappa: \name{W}\rest\alpha\in \spl(T)\ \wedge\
    \name{W}(\alpha)\in \suc_T(\name{W}\rest\alpha)\}$ is a club''.
  \item If $\ell=3$, then for some club $C\subseteq \kappa$ we have 
    $T\forces_{\bbQ^2_{f,\bar{F}}}$`` $\big(\forall\alpha<\kappa\big)\big(
    \name{W}\rest\alpha\in \spl(T)\ \wedge\ \name{W}(\alpha)\in
    \suc_T(\name{W} \rest\alpha)\big)$ ''. 
  \item If $\ell=4$ then $T\forces_{\bbQ^4_{f,\bar{F}}}$``
    $\big(\forall\alpha\geq \lh(\mrot(T))\big)\big(
    \name{W}(\alpha)\in \suc_T(\name{W}\rest\alpha)\big)$ ''. 
  \end{enumerate}
\end{observation}

\begin{problem}
  If $\bar{\kappa},\bar{F}$ are as in Definition \ref{boundedPQ}, how
  different are $\bbQ^2_{\bar{\kappa},\bar{F}}$ and
  $\bbQ^3_{\bar{\kappa},\bar{F}}$ ?
\end{problem}

\section{Forcing for  $\leq^\cl$ and $\cancel{\ni^{\rm stat}}$}

\begin{definition}
  \label{defavoid}
Suppose $b,h\in\cR^+_\kappa\cup\{\bar{\kappa}\}$, $h\leq b$. We define
$\bar{B}(h,b)= \big \langle B^{h,b}_t: t\in\bigcup\limits_{\alpha<\kappa}
\prod\limits_{\xi<\alpha} b(\xi) \big\rangle$, where
\[B^{h,b}_t=\big\{A\subseteq b(\lh(t)):\big|b(\lh(t))\setminus A\big|
  <h(\lh(t))\big\}.\] 
\end{definition}

\begin{remark}
  \label{remavoid}
  Let  $b,h\in\cR^+_\kappa\cup\{\bar{\kappa}\}$, $h\leq b$.
  \begin{enumerate}
  \item Then $\kappa,b$ and $\bar{B}(h,b)$ satisfy the assumptions of
    Definition \ref{boundedPQ}.
\item The forcing $\bbQ^*_b$ introduced in Definition \ref{boudQ4} is 
$\bbQ^4_{b,\bar{B}(b,b)}$ and the forcing notion $\bbQ^h_b$ is
$\bbQ^4_{b,\bar{B}(h,b)}$. 
  \end{enumerate}
\end{remark}

\begin{theorem}
  \label{silly}
  Suppose $\kappa$ is strongly inaccessible and
  $2^\kappa=\kappa^+$. Assume also that $\diamondsuit_\kappa$   holds
  true. Then there is a ${<}\kappa$--strategically complete $\kappa^{++}$--cc
  $\kappa$-proper forcing notion $\bbP$ such that 
  \[\forces_{\bbP} \mbox{``} \big(\forall b,h\in
    \cR_\kappa^+\cup\{\bar{\kappa}\} \big)\big(h\leq
    b\ \Rightarrow\  \gb_\kappa^{b,h}(\cancel{\ni^\infty})=
    \gb_\kappa^{b,h}(\cancel{\ni^{\rm stat}})=\gb_\kappa(\leq^*) = 2^\kappa
    = \kappa^{++}\big)\mbox{''.}\]     
\end{theorem}

\begin{proof}
  Assume $\diamondsuit _\kappa$ and let $\cS$ be as in Context
  \ref{context}(2,5). Then the assumptions are preserved by
  ${<}\kappa$--strategically complete $\kappa$--proper forcing
  notions. Furthermore, in any extension by such forcings,  posets of the
  type $\bbQ^\ell_{b,B(h,b)}$ (for $\ell\in \{2,3,4\}$ and
  $b,h\in\cR^+_\kappa \cup\{\bar{\kappa}\}$, $h\leq b$) are purely 
  sequentially${}^{+}$ proper over $(D,\cS)$--semi-diamonds (remember
  Proposition \ref{getpurity}).

  Set up a $\kappa$--support iteration $\langle \bbP_\xi,\name{\bbQ}_\xi:
   \xi<\kappa^{++}\rangle$ such that
\begin{enumerate}
\item[(a)]  for every $\xi<\kappa^{++}$, for some $\bbP_\xi$--names
  $\name{b}_\xi, \name{h}_\xi$  for elements of $\cR_\kappa^+\cup
  \{\bar{\kappa}\}$ (in $\bV^{\bbP_\xi}$) we have $\forces_{\bbP_\xi}
  \name{\bbQ}_\xi= \bbQ^{\name{h}_\xi}_{\name{b}_\xi}$; 
\item[(b)] for every $\bbP_{\kappa^{++}}$--names $\name{b},\name{h}$ for elements of
  $\cR_\kappa^+\cup\{\bar{\kappa}\}$ (in $\bV^{\bbP_{\kappa^{++}}}$) such
  that $\name{h}\leq \name{b}$ there are cofinally many
  $\xi<\kappa^{++}$ such that $\forces_{\bbP_{\kappa^{++}}}$``
  $\name{b}=\name{b}_\xi$ and $\name{h}=\name{h}_\xi$ ''.  
\end{enumerate}
It is possible to carry out the construction, because demand (a) and Theorem
\ref{n3.7} guarantee that each $\bbP_\xi$ will be
strategically  $({<}\kappa)$--complete and $\kappa$--proper. Thus by 
\cite[Theorem 1.3]{RoSh:1001} we know that $\forces_{\bbP_\xi}
2^\kappa=\kappa^+$ (for each $\xi<\kappa^{++}$) and $\bbP_{\kappa^{++}}$ 
is $\kappa^{++}$--cc. The rest of the argument follows the usual
pattern. 
 \end{proof}

  \begin{theorem}
    \label{re125one}
    Assume that $\kappa$ is strongly inaccessible cardinal,
    $2^\kappa=\kappa^+$. Then there is a strategically
    $({<}\kappa)$--complete, $\kappa$--proper, $\kappa^{++}$--cc forcing notion
    $\bbP$ such that
\[\forces_{\bbP} \mbox{``} \gd_\kappa(\leq^*)=\kappa^+\mbox { and }
  \big(\forall b,h\in \cR_\kappa^+\big)\big(h\leq b\ \Rightarrow\
  \gb_\kappa^b( \leq^\cl)= \gb_\kappa^{b,h}(\cancel{\ni^{\rm stat}}) =
  2^\kappa= \kappa^{++}\big)\mbox{''.}\]   
  \end{theorem}

  \begin{proof}
    For $b,h\in \cR_\kappa^+$, $h\leq b$, let $\bar{B}(h,b)$ be as
    introduced in Definition \ref{defavoid}.  To simplify notation, let
    $\bbQ^2_{b,\bar{B}(h,b)}$ be denoted by $\bbQ^2(h,b)$. By  Proposition
    \ref{q2isAb} the forcing notion  $\bbQ^2(h,b)$ is reasonably A--bounding
    over $\bar{\kappa}$. 
    
Using suitable bookkeeping arguments construct a $\kappa$--support iteration 
$\bar{\bbQ}=\langle\bbP_\xi,\name{\bbQ}_\xi:\xi<\kappa^{++}\rangle$ such
that
\begin{enumerate}
\item[(a)]  for every $\xi<\kappa^{++}$, for some $\bbP_\xi$--names
  $\name{b}_\xi,\name{h}_\xi$ for elements of $\cR_\kappa^+$ (in
  $\bV^{\bbP_\xi}$) we have $\forces_{\bbP_\xi}$`` $\name{h}_\xi\leq
  \name{b}_\xi\ \wedge\  \name{\bbQ}_\xi= \bbQ^2(\name{h}_\xi,
  \name{b}_\xi)$ ''; 
\item[(b)] for all $\bbP_{\kappa^{++}}$--names $\name{b},\name{h}$ for
  elements of $\cR_\kappa^+$ (in $\bV^{\bbP_{\kappa^{++}}}$) there are
  cofinally many $\xi<\kappa^{++}$ such that
  $\forces_{\bbP_{\kappa^{++}}}$`` $\name{h}\leq \name{b}\ \Rightarrow\ 
  (\name{h}=\name{h}_\xi\ \wedge\ \name{b}=\name{b}_\xi)$ ''. 
\end{enumerate}
It is possible to carry out the construction, because demand (a) and 
Proposition \ref{q2isAb} imply that the iterands are reasonably A--bounding
over $\bar{\kappa}$. Therefore Theorem \ref{verA} guarantees that each
$\bbP_\xi$ will be strategically $({<}\kappa)$--complete and
$\kappa$--proper. Thus by  \cite[Theorem 1.3]{RoSh:1001} we know that
$\forces_{\bbP_\xi}  2^\kappa=\kappa^+$ (for each $\xi<\kappa^{++}$) and
$\bbP_{\kappa^{++}}$ is $\kappa^{++}$--cc. Since $\bbP_{\kappa^{++}}$ will
be $\leq^*$--bounding (by Theorem \ref{verA}), we will have 
$\forces_{\bbP_{\kappa^{++}}} \gd_\kappa(\leq^*)=\kappa^+$.
On the other hand, forcing with $\bbQ^2(h,b)$ adds a $\cancel{\ni^{\rm
    stat}}$--dominating element of  $\prod\limits_{\alpha<\kappa}
b(\alpha)$ (with respect to ${\rm Loc}^{b,h}_\kappa$). Therefore, demand (b)
implies

$\displaystyle \forces_{\bbP_{\kappa^{++}}}  2^\kappa=\kappa^{++}\mbox { and }
  \big(\forall h,b\in \cR_\kappa^+\big) \big(h\leq b\ \Rightarrow
  \gb_\kappa^b(\leq^\cl) =\gb^{b,h}_\kappa(\cancel{\ni^{\rm stat}})=
  \kappa^{++}\big) $.  
  \end{proof}

  \begin{theorem}
    \label{oneb}
    Assume that $\kappa$ is strongly inaccessible cardinal,
    $2^\kappa=\kappa^+$ and $b,h\in \cR_\kappa^+$, $h\leq b$. Let
    $d\in\cR_\kappa$ be given by $d(\alpha)=\Big(b(\alpha)^{|\alpha|}
    \Big)^+$ for  $\alpha<\kappa$. Then there is a strategically
    $({<}\kappa)$--complete, $\kappa$--proper, $\kappa^{++}$--cc
    forcing notion $\bbP$ such that
    \[\begin{array}{r}
      \forces_{\bbP} \mbox{`` }\big(\forall c\in\cR_\kappa\cup
        \{\bar{\kappa}\} \big)\big(c\geq d\Rightarrow  \gd_\kappa^{c,d}
        (\in^\cl)= \gd_\kappa^{c,d}(\in^*)=  \gd^c_\kappa(\leq^*)=
        \gd_\kappa^c(\leq^\cl)=\kappa^+\big) \mbox { and }\\
        \gb_\kappa^b(\leq^\cl)=\gb^{b,h}_\kappa(\cancel{\ni^{\rm stat}})=
        2^\kappa= \kappa^{++}\mbox{ ''.} 
        \end{array}\]   
  \end{theorem}

  \begin{proof}
Let $\bar{B}(b,b)$ and $\bbQ^2_{b,\bar{B}(b,b)}=\bbQ^2_{b,b}$ be as in the 
proof of Theorem \ref{re125one}. Since $\big|\prod\limits_{\xi\leq
  \alpha} b(\xi)\big|<d(\alpha)$ for all $\alpha$, the forcing notion
$\bbQ^2_{b,b}$ is reasonably A--bounding over $d$ (by Proposition
\ref{getAbound}). Also, the function $d$ satisfies the assumption (b) of
Theorem \ref{verA}. Therefore, the limit $\bbP$ of the
$\kappa$--support iteration of $\bbQ^2_{b,b}$ of length $\kappa^{++}$ will
satisfy
\[\forces_\bbP\big(\forall x\in \bairek\big)\big(\exists \psi\in
  {\rm Loc}^{\bar{\kappa},d}_\kappa\cap \bV\big)\big(\forall\alpha<\kappa
  \big) \big (x(\alpha)\in \psi(\alpha)\big).\] 
Consequently, $\bbP$ will have all the required properties. 
  \end{proof}

  \begin{problem}
    \begin{enumerate}
    \item What is the value of $\gb_\kappa^b(\leq^*)$ in the model
      constructed in the proof of Theorem \ref{re125one} for (some)
      $b\in\cR^+_\kappa$?  
\item Suppose $b,h\in\cR_\kappa^+$ and for each $\alpha<\kappa$ we have
  $\big(2^{h(\alpha)}\big)^+<b(\alpha)$. Is it consistent that
  $\gb^{b,h}_\kappa(\cancel{\ni^{\rm stat}})>\gb^b_\kappa(\leq^\cl)$?  
     \item It was shown by Cummings and Shelah \cite{CuSh:541} that, for a
       strongly inaccessible $\kappa$,
       $\gd_\kappa(\leq^*)=\gd_\kappa(\leq^\cl)$ and
       $\gb_\kappa(\leq^*)=\gb_\kappa(\leq^\cl)$. Is the parallel true for
       the bounded versions $\gd^b_\kappa,\gb^b_\kappa$? Possibly for larger
       $\kappa$, like weakly compact, and $b$ increasing fast enough?
    \end{enumerate}
  \end{problem}

  \section{Forcing for $\in^*$ and/or $\in^\cl$}
In \cite[Question 129]{vdV25} van der Vlugt asked:  Do there exist
nontrivial  $b,h,b',h'\in\bairek$ such that, consistently,
$\gb^{b,h}_\kappa(\in^*)<\gb^{b',h'}_\kappa(\in^*)$? 
Brendle, Brooke-Taylor, Friedman, and Montoya \cite[Question
  71]{BBFM-17} asked if it is consistent that $\gb_\kappa^{{\rm
      id}^+}(\in^*) <  \gb_\kappa^{{\rm pow}^+}(\in^*)$.
  Naturally, for this kind of problems one would like to consider
  forcing notions adding  $\varphi\in {\rm  Loc}^{b,h}_\kappa$ which
  will localize all elements of  $\prod\limits_{\alpha<\kappa}
  b(\alpha)\cap \bV$. Then we may consider the following approach.

\begin{definition}
  \label{foraddsl}
  Let $b,h\in {}^\kappa(\kappa+1)$ be such that $h(\alpha)\leq
  b(\alpha)$ for all $\alpha<\kappa$. We define a function $f^{b,h}=f$
  on $\kappa$ and filters $E^\alpha_{b,h}=E^\alpha$ on $f(\alpha)$
  (for $\alpha<\kappa$) as follows.  Fix $\alpha<\kappa$. Let
  $f(\alpha)=\big[b(\alpha)\big]^{\textstyle  {<}h(\alpha)}$. For
  $a\in \big[b(\alpha)\big]^{\textstyle  {<}h(\alpha)}$ let
  $Z^{b,h}(a)=\big\{A\in \big[b(\alpha)\big]^{\textstyle 
    {<}h(\alpha)}: a\subseteq A\big\}$. Put
  \[E^\alpha=\big\{Z\subseteq \big[b(\alpha)\big]^{\textstyle
    {<}h(\alpha)}: \big(\exists a\in \big[b(\alpha)\big]^{\textstyle
    {<}h(\alpha)}\big)\big(Z^{b,h}(a)\subseteq Z\big)\big\}.\]
\end{definition}

With $b,h$ and $f(\alpha), E^\alpha$ as in \ref{foraddsl}, we have: if
$a_i\in \big[b(\alpha)\big]^{\textstyle  {<}h(\alpha)}$ for
$i<i^*<\cf(h(\alpha))$ and $a=\bigcup\limits_{i<i^*} a_i$, then $a\in
\big[b(\alpha)\big]^{\textstyle  {<}h(\alpha)}$ and
$Z^{b,h}(a)=\bigcap\limits_{i<i^*} Z^{b,h}(a_i)$. Therefore:

\begin{observation}
  \label{compfil}
Let $b,h$ and $f(\alpha), E^\alpha$ be as in \ref{foraddsl}. Then
$E^\alpha$ is a ${<}\cf(h(\alpha))$--complete filter of subsets of
$f(\alpha)$.
\end{observation}

\begin{definition}
\label{Pfor}
  Let $\kappa$ be an inaccessible cardinal and let $b,h\in
  {}^\kappa(\kappa+1)$ be such that $h(\alpha)\leq b(\alpha)$ for all
  $\alpha<\kappa$. Let $f^{b,h}=f$ and $E^\alpha_{b,h}=E^\alpha$ be as 
  defined in  \ref{foraddsl}. Let $\bar{E}=\bar{E}^{b,h}=\langle
  E_t^{b,h}: t\in\bigcup\limits_{\alpha<\kappa} \prod\limits_{\xi
    <\alpha} f^{b,h}(\xi)\rangle$ where
  $E^{b,h}_t=E_{b,h}^{\lh(t)}$. We define forcing notions
  $\bbP_{\rm loc}^\cl(b,h)$  and $\bbP^*_{\rm loc}(b,h)$:
  \[\bbP_{\rm loc}^\cl(b,h)=\bbQ^2_{f,\bar{E}}\quad \mbox{ and }\quad
    \bbP^*_{\rm loc}(b,h) = \bbQ^4_{f,\bar{E}}.\]
\end{definition}

\begin{proposition}
\label{Pisprop}
  Let $\kappa$ be an inaccessible cardinal and let $b,
  h\in\cR^+_\kappa$ be such that $h(\alpha)\leq b(\alpha)$ for all
  $\alpha<\kappa$. 
  \begin{enumerate}
\item If $d(\alpha)> \big |\prod\limits_{\xi\leq\alpha}
  b(\xi)^{<h(\xi)}\big|$, then the forcing notion $\bbP_{\rm
    loc}^\cl (b,h)$ is reasonably A--bounding  over $d$.
\item The forcing notions $\bbP_{\rm loc}^\cl(\bar{\kappa},h)$,
  $\bbP^*_{\rm loc}(\bar{\kappa},h)$ and $\bbP^*_{\rm  loc}(b,h)$,
  $\bbP_{\rm loc}^\cl(b,h)$ are purely sequentially$^+$ proper over semi
  diamonds (see Definition \ref{Def3.6}).
\item If $\name{W}$ is the generic function on $\kappa$ added by
  $\bbP^*_{\rm loc}(b,h)$ (i.e., $T\forces_{\bbP^*_{\rm
      loc}(b,h)} \mrot(T)\vtl \name{W}$), then
  \[\forces_{\bbP^*_{\rm loc}(b,h)}\mbox{`` }\name{W}\in {\rm
      Loc}^{b,h}_\kappa\mbox{ and } \big(\forall x\in
    \prod_{\alpha<\kappa} b(\alpha)\cap\bV\big)  \big(x\in^* \name{W}
    \big)\mbox{ '' .}\]
  Similarly for $\bar{\kappa},h$ in place of $b,h$.
\item If $\name{W}$ is the generic function on $\kappa$ added by
  $\bbP_{\rm loc}^\cl(b,h)$, then 
  \[\forces_{\bbP_{\rm loc}^\cl(b,h)}\mbox{`` }\name{W}\in {\rm
      Loc}^{b,h}_\kappa\mbox{ and }    \big(\forall x\in
    \prod_{\alpha<\kappa} b(\alpha)\cap\bV\big)
    \big(x\in^\cl\name{W}\big) \mbox{ '' .}\] 
Similarly for $\bar{\kappa},h$ in place of $b,h$. 
  \end{enumerate}
\end{proposition}

\begin{proof}
  (1)\quad Follows from Observation \ref{compfil} and Proposition
  \ref{getAbound}. 

  \medskip

  \noindent (2)\quad Follows from Proposition \ref{getpurity}.
  \medskip

  \noindent (3,4)\quad Should be clear (remember Observation \ref{genseq}).
\end{proof}

\begin{remark}
  The assumption that $h\in\cR^+_\kappa$ (so $\alpha<h(\alpha)$) in
  Proposition \ref{Pisprop} is really necessary. In our approach we
  need it for strategic $({<}\kappa)$--completeness of our
  forcings. But the reason is deeper: by \cite[Lemma 31]{vdV25} in
  other cases the value of $\gb^{b,h}_\kappa(\in^*)$ would be at most
  $\kappa$, so we could not force it to be large.
\end{remark}

\begin{theorem}
  \label{allbig}
    Let $\kappa$ be a strongly inaccessible cardinal,  $2^\kappa=
    \kappa^+$. Assume $\diamondsuit_\kappa$. Then there is a strategically
    $({<}\kappa)$--complete,  $\kappa$--proper, $\kappa^{++}$--cc forcing
    notion $\bbP$ such that 
    \[\begin{array}{ll}
        \forces_{\bbP} \mbox{``}&\mbox{for every $b,h\in\cR^+_\kappa
                                  \cup\{\bar{\kappa}\}$  such that $h\leq b$  we have}\\  
                       &\gb_\kappa^{b,h}(\in^\cl)=\gb^{b,h}_\kappa(\in^*)=\gb^h_\kappa(
                       \in^*)=\gb^h_\kappa(\in^\cl) =2^\kappa= \kappa^{++}\mbox{ ''.}
    \end{array}\]   
  \end{theorem}

  \begin{proof}
The forcing notion $\bbP^*_{\rm loc}(\bar{\kappa},{\rm id}^+)$ is purely
sequentially$^+$ proper over semi diamonds (see \ref{Pisprop}(2)), so its
$\kappa$--support iterations are $\kappa$--proper (by Theorem
\ref{n3.7}). Also $\bbP^*_{\rm loc}(\bar{\kappa},{\rm id}^+)$ adjoins a
slalom $\name{W}\in {\rm Loc}^{\bar{\kappa},{\rm id}^+}_\kappa$ such that
$x\in^*\name{W}$ for every $\kappa\in\bairek$ from the ground model. This
slalom can be restricted to $\prod\limits_{\xi<\kappa} b(\xi)$ for any
$\xi\in \cR^+_\kappa$. Thus for every $h,b\in \cR^+_\kappa$, $h\leq b$,
forcing with $\bbP^*_{\rm loc}(\bar{\kappa},{\rm id}^+)$ adds a slalom in
${\rm Loc}^{b,h}_\kappa$ $\in^*$--capturing all elements of
$\prod\limits_{\xi<\kappa} b(\xi)\cap\bV$. Consequently, under the
assumption of $\diamondsuit_\kappa$, $\kappa$--support iteration of length
$\kappa^{++}$ of forcing notion $\bbP^*_{\rm loc}(\bar{\kappa},{\rm id}^+)$ will
have the required properties. 
  \end{proof}

  \begin{theorem}
    \label{bhloc}
Let $\kappa$ be a strongly inaccessible cardinal,
$2^\kappa=\kappa^+$. Suppose $b,h\in\cR^+_\kappa$, $h\leq b$ and for
$\alpha<\kappa$ let
\[d(\alpha)=\Big(\Big|\prod_{\xi\leq\alpha}
  b(\xi)^{<h(\xi)}\Big|^{|\alpha|} \Big)^+\]
(so $d\in\cR^+_\kappa$). Then there is a strategically
$({<}\kappa)$--complete, $\kappa$--proper, $\kappa^{++}$--cc forcing
notion $\bbP$ such that
\[\forces_\bbP\mbox{`` }\gd_\kappa(\leq^*)=\gd_\kappa^{\bar{\kappa},d}(
  \in^*) =\kappa^+\ \mbox{ and }\ \gb^{b,h}_\kappa(\in^\cl)=2^\kappa=
  \kappa^{++}\mbox{ ''.}\]
\end{theorem}

\begin{proof}
Since $\big|\prod\limits_{\xi\leq
  \alpha} b(\xi)^{<h(\xi)}\big|<d(\alpha)$ for all $\alpha$, the forcing notion
$\bbP_{\rm loc}^\cl(b,h)$ is reasonably A--bounding over $d$ (by Proposition
\ref{q2isAb}). The function $d$ satisfies the assumption (b) of
Theorem \ref{verA}, so the limit $\bbP$ of the $\kappa$--support
iteration of $\bbP_{\rm loc}^\cl(b,h)$ of length $\kappa^{++}$ will 
be reasonably ${\bf a}$--bounding over $d$. Hence, $\bbP$ will have
all the required properties. 
\end{proof}

\begin{theorem}
  \label{onebig}
    Let $\kappa$ be a strongly inaccessible cardinal, $2^\kappa=
    \kappa^+$. Then there is a strategically $({<}\kappa)$--complete,
    $\kappa$--proper, $\kappa^{++}$--cc forcing notion $\bbP$ such that
    \[\begin{array}{ll}
        \forces_{\bbP} \mbox{``}&\gd_\kappa(\leq^*)=\kappa^+\ \mbox{ and for
                                  every $b,h\in\cR^+_\kappa$ such that $h\leq b$  we have}\\  
&\gb_\kappa^{b,h}(\in^\cl)=\gb^{b,h}_\kappa(\in^*) =2^\kappa= \kappa^{++}\mbox{ ''.}
    \end{array}\]   
  \end{theorem}

  \begin{proof}
    Using suitable bookkeeping arguments construct a $\kappa$--support iteration
$\bar{\bbQ}=\langle\bbP_\xi,\name{\bbQ}_\xi:\xi<\kappa^{++}\rangle$ such
that
\begin{enumerate}
\item[(a)]  for every $\xi<\kappa^{++}$, for some $\bbP_\xi$--names
  $\name{b}_\xi, \name{h}_\xi$ for functions in $\cR_\kappa^+$ such that
  $\name{h}_\xi\leq \name{b}_\xi$ (in $\bV^{\bbP_\xi}$) we have
  $\forces_{\bbP_\xi}  \name{\bbQ}_\xi=  \bbP^\cl_{\rm
    loc}(\name{b}_\xi,\name{h}_\xi)$;  
\item[(b)] for all $\bbP_{\kappa^{++}}$--names $\name{b},\name{h}$
  for elements of $\cR_\kappa^+$ with $\name{h}\leq \name{b}$ (in 
  $\bV^{\bbP_{\kappa^{++}}}$) there are cofinally many 
  $\xi<\kappa^{++}$ such that $\forces_{\bbP_{\kappa^{++}}}
  \name{b}=\name{b}_\xi\ \wedge \name{h}=\name{h}_\xi$. 
\end{enumerate}
It is possible to carry out the construction, because demand (a) and 
Proposition \ref{getAbound} imply that the iterands are reasonably
A--bounding over $\bar{\kappa}$. Therefore Theorem \ref{verA}
guarantees that each $\bbP_\xi$ will be strategically $({<}\kappa)$--complete and
$\kappa$--proper. Thus by  \cite[Theorem 1.3]{RoSh:1001} we know that
$\forces_{\bbP_\xi}  2^\kappa=\kappa^+$ (for each $\xi<\kappa^{++}$) and
$\bbP_{\kappa^{++}}$ is $\kappa^{++}$--cc. Also, $\bbP_{\kappa^{++}}$ will
be $\leq^*$--bounding (by \ref{verA}), so $\forces_{\bbP_{\kappa^{++}}}
\gd_\kappa(\leq^*)=\kappa^+$.  

It follows from Proposition \ref{Pisprop}(4) that in the extension via
$\bbP_{\kappa^{++}}$, for any $b,h\in\cR^+_\kappa$ with $h\leq b$,
every family of $\kappa^+$ many functions from
$\prod\limits_{\alpha<\kappa} b(\alpha)$, will be $\in^\cl$--localized
by one slalom from ${\rm  Loc}^{b,h}_\kappa$. Also every family of
$\kappa^+$ many functions from $\bairek$ will be $\in^*$--localized
by one slalom from ${\rm  Loc}^{\bar{\kappa},{\rm id}^+}_\kappa$. Hence

$\forces_{\bbP_{\kappa^{++}}}\big(b,h\in\cR^+_\kappa\big)\big(h\leq b\
  \Rightarrow \ \gb_\kappa^{b,h}(\in^\cl)=\kappa^{++}\big)$.
  \end{proof}

\bigskip

  Similarly to the case of the eventual dominance $\leq^*$, we have
  evidence that some of the questions involving bounding and/or
  dominating numbers for $\in^*$ may require larger cardinals.

\begin{proposition}
\label{Againunb}
Let $b,h\in {}^\kappa(\kappa+1)$, $\omega\leq h\leq b$. Then the
forcing notion  $\bbP^*_{\rm loc}(b,h)$ adds a $\leq^*$--unbounded
function in ${}^\kappa\kappa$. 
\end{proposition}

\begin{proof}
  Same as for Proposition \ref{addunb}, using the following observation.
  \begin{claim}
    \label{cl5}
For every $\alpha<\kappa$, there are disjoint sets
$Z^\alpha_0,Z^\alpha_1\in\big(E^\alpha_{b,h}\big)^+$. 
\end{claim}

\begin{proof}[Proof of the Claim]
  Enumerate $\big[b(\alpha)\big]^{\textstyle {<}h(\alpha)}=\langle
  a_\xi:\xi<\mu\rangle$, where $\mu$ is the cardinality of
  $\big[b(\alpha)\big]^{\textstyle {<}h(\alpha)}$. By induction on
  $\xi<\mu$ choose $e^0_\xi,e^1_\xi\in \big[b(\alpha)\big]^{\textstyle
    {<}h(\alpha)}$ so that for all $\xi,\zeta<\mu$:
  \[a_\xi\subseteq e^0_\xi\cap e^1_\xi\quad \mbox{ and }\quad
    e^0_\xi\neq e^1_\zeta.\]
 Then let $Z^\alpha_\ell=\{e^\ell_\xi:\xi<\mu\}$.  
\end{proof}
\end{proof}

\begin{proposition}
  \label{Againwc}
Assume $\kappa$ is weakly compact, $h,b\in\cR^+_\kappa$ and $h\leq b$. Then
the forcing notion $\bbP^*_{\rm loc}(b,h)$ does not add a
$\leq^*$--dominating function in ${}^\kappa\kappa$.    
\end{proposition}

\begin{proof}
  Same as for Proposition  \ref{wcomp}.
\end{proof}

\begin{theorem}
  \label{re129one}
Assume $\kappa$ is strongly inaccessible not Mahlo and suppose 
$\diamondsuit_\kappa$ holds true. Let $b,h\in\cR^+_\kappa$, $h\leq
b$. Then the forcing notion $\bbP^*_{\rm loc}(b,h)$ adds a
$\leq^*$--dominating function in $\bairek$. 
\end{theorem}

\begin{proof}
  The arguments are essentially the same as for Proposition \ref{bad}.
  Let $C\subseteq \kappa$ be a club consisting of singular cardinals
  and let $\langle \varphi_\delta:\delta\in C\rangle$ be such that
  \begin{enumerate}
  \item[$(\boxtimes)_1$] $\varphi_\delta\in
    \prod\limits_{\alpha<\delta} [b(\alpha)]^{\textstyle {<}h(\alpha)}$
    for $\delta\in C$,
  \item[$(\boxtimes)_2$] for each $\varphi\in {\rm
      Loc}^{b,h}_\kappa$, the set $\{\delta\in C:\varphi\rest
    \delta=\varphi_\delta\}$ is stationary.     
  \end{enumerate}

  \begin{claim}
    \label{cl6}
    For each $S\in \bbP^*_{\rm loc}(b,h)$ and $\alpha<\kappa$,
    $\{f_\delta:\delta\in C\setminus \alpha\}\cap S\neq \emptyset$.
  \end{claim}

  \begin{claim}
    \label{cl7}
    If $\alpha<\beta$ are from $C$, then there is $\psi\in
    \prod\limits_{\xi\in [\alpha,\beta)} \big[b(\xi)\big]^{\textstyle
        {<}h(\xi)}$ such that
      \[\big(\forall \delta\in C\cap (\alpha,\beta]\big)\big(\exists
        \xi\in [\alpha,\delta)\big)\big(\varphi_\delta(\xi)\subseteq
        \psi(\xi)\Big).\]
  \end{claim}

  \begin{proof}[Proof of the Claim]
    Induction on the order type of $C\cap (\alpha,\beta]$, very much
    like in \ref{cl2}. At the limit stage, if $\beta=\sup(C\cap
    \beta)$, since  $\cf(\beta)<\beta$, we may choose an increasing
    continuous sequence $\langle \gamma_i:i<\cf(\beta)\rangle\subseteq
    C\cap (\alpha,\beta)$ with $\cf(\beta)<\gamma_0$ and
    $\sup\limits_{i<\cf(\beta)} \gamma_i=\beta$. By the inductive
    hypothesis we may find functions $\psi_0\in\prod\limits_{\xi\in
      [\alpha,\gamma_0)} \big[b(\xi)\big]^{\textstyle {<}h(\xi)}$
    and $\psi_{i+1}  \in\prod\limits_{\xi\in [\gamma_i,\gamma_{i+1})}
    \big[b(\xi)\big]^{\textstyle {<}h(\xi)}$ 
 (for $i<\cf(\beta)$) such that
 \[\begin{array}{l}
   \big(\forall \delta\in C\cap (\alpha,\gamma_0]\big) \big(\exists\xi\in
   [\alpha,\delta) \big) \big(\varphi_\delta(\xi)\subseteq \psi_0(\xi)
     \big)\quad \mbox{ and}\\
\ \\
     \big(\forall \delta\in C\cap (\gamma_i,\gamma_{i+1}]\big)
     \big(\exists\xi\in  [\gamma_i,\delta) \big)
     \big(\varphi_\delta(\xi) \subseteq \psi_{i+1}(\xi)
     \big)\quad\mbox{  for  all }i<\cf(\beta).  
   \end{array}\]
Define
$g\in\prod\limits_{\xi\in [\alpha,\beta)} \big[b(\xi)\big]^{\textstyle {<}
  h(\xi)}$ by 
  \begin{quotation}
\noindent  $\psi\rest [\alpha,\gamma_0)=\psi_0$, $\psi\rest
(\gamma_0,\gamma_1)=\psi_1\rest (\gamma_0,\gamma_1)$ and

\noindent $\psi(\gamma_0)=\psi_1(\gamma_0)\cup
\bigcup\limits_{0<i<\cf(\beta)} \varphi_{\gamma_i}(\gamma_0)$ (note that 
$\cf(\beta)<h(\gamma_0) =\cf(h(\gamma_0))$, and 

\noindent $\psi\rest [\gamma_i,\gamma_{i+1})=\psi_{i+1}$ for
    $0<i<\cf(\beta)$. 
  \end{quotation}
Then $\psi$ will have the desired property.  
  \end{proof}

Let $\name{W}$ be the generic function on $\kappa$ added by
$\bbP^*_{\rm loc}(b,h)$, so $\forces_{\bbP^*_{\rm loc}(b,h)}
\name{W}\in {\rm Loc}^{b,h}_\kappa$. Let $\name{A}$ be such that
\[\forces_{\bbP^*_{\rm loc}(b,h)} \mbox{`` }\name{A}=\{\delta\in C:
  \varphi_\delta \subseteq \name{W}\}\in [\kappa]^{\textstyle
    \kappa}  \mbox{ '',}\] 
and let 
\[\forces_{\bbP^*_{\rm loc}(b,h)} \mbox{`` }\name{\tau}(\xi)\mbox{ is
    the $(\xi+1)^{\rm st}$ element of }\name{A} \mbox{ '',}\] 

\begin{claim}
  \label{cl8}
$\displaystyle \forces_{\bbP^*_{\rm loc}(b,h)} \mbox{`` }\big(\forall
x\in\bairek \cap\bV\big)\big(x\leq^* \name{\tau}\big)$ ''.
\end{claim}

\begin{proof}[Proof of the Claim]
Let $x\in \bairek$ and $S\in\bbP^*_{\rm loc}(b,h)$. Choose an
increasing continuous sequence $\langle
\gamma_i:i<\kappa\rangle\subseteq C$ such that $x(\xi)<\gamma_i$ for
all $\xi<\gamma_i$, $i<\kappa$. Use Claim \ref{cl7} to build $\psi\in
{\rm Loc}^{b,h}_\kappa$ such that
\begin{enumerate}
\item[$(\boxtimes)_3$] $\big(\forall i<\kappa\big)\big( \delta\in
  C\cap(\gamma_i, \gamma_{i+1}]\big)\big(\exists \xi\in
  [\gamma_i,\delta)\big) \big(\varphi_\delta(\xi)\subseteq
  \psi(\xi)\big)$.
\end{enumerate}
By Definition \ref{foraddsl} of the filters $E^\alpha_{b,h}$, for each
$t\in S$ we may choose a set $M_t\subseteq \suc_S(t)$ such that
$M_t\in E^{\lh(t)}_{b,h}$ and for each $Z\in M_t$ we have
$Z\setminus \psi(\lh(t))\neq \emptyset$. Next, build a condition $T\in \bbP^*_{\rm
    loc}(b,h)$ stronger than $S$ and such that $\mrot(T)=\mrot(S)$
  and $\suc_T(t)=M_t$ for each $t\in T$. Then $T\forces
  \name{A}\setminus \lh(\mrot(T))\subseteq \{\gamma_j:j<\kappa\mbox{
    is limit }\}$. Hence easily $T\forces
  \big(\forall\xi>\lh(\mrot(T))\big)\big(
  x(\xi)<\name{\tau}(\xi)\big)$.   
\end{proof}
\end{proof}

\begin{corollary}
  \label{re129two}
Suppose $\kappa$ is strongly inaccessible not Mahlo and
$\diamondsuit_\kappa$ holds true. Let $b,h\in\cR^+_\kappa$, $h\leq
b$. Then $\gd_\kappa(\leq^*)\leq \gd_\kappa^{b,h}(\in^*)$.
\end{corollary}

\begin{proof}
One could extract a proof from the arguments in \ref{re129one} in a
manner similar to the arguments presented in \ref{re125two}. However,
the Corollary actually follows from Theorem \ref{re125two}. Under
current assumptions we have:
\[\gd_\kappa(\leq^*)\leq \gd^b_\kappa(\leq^*)=
  \gd_\kappa^{b,b}(\in^*)\leq \gd^{b,h}_\kappa(\in^*).\]
\end{proof}

\begin{problem}
  \begin{enumerate}
  \item Let $h(\alpha)=\big(2^{|\alpha|}\big)^{++}<b(\alpha)<\kappa$ for all
    $\alpha<\kappa$. Can a $\kappa$--support iteration of $\bbP^*_{\rm loc}(
    b,h)$ add a slalom in ${\rm Loc}^{\bar{\kappa},{\rm id}^+}_\kappa$
    $\in^*$--localizing all elements of $\bairek\cap\bV$?
\item Suppose $b,c\in\cR^+_\kappa$, $\big(2^b\big)^{+++}\leq c$. Can a
  $\kappa$--support iteration of $\bbP^*_{\rm loc}(b,{\rm id}^+)$ add a
  slalom in ${\rm Loc}^{c,{\rm id}^+}_\kappa$   $\in^*$--localizing all
  elements of $\bV\cap\prod\limits_{\xi<\kappa} c(\xi)$? 
  \end{enumerate}
\end{problem}

\section{Forcing for $\cancel{=^{\rm stat}}$}

\begin{definition}
  \label{NRfor}
  Let $b\in {}^\kappa(\kappa+1)$ and $\bar{2}\leq b$.  For $\alpha<\kappa$
  set $F^b_\alpha=\big\{b(\alpha)\big\}$ --- we will treat $F^b_\alpha$ as a
  filter of subsets of $b(\alpha)$. We let
  $\bar{F}^b=\big\langle F^b_t:t\in \bigcup\limits_{\alpha<\kappa}
  \prod\limits_{\xi<\alpha} b(\xi)\big\rangle$, where $F^b_t=F^b_{\lh(t)}$.

The forcing notion $\bbQ^2_{b,\bar{F}^b}$ will be denoted $\bbD_b$.
\end{definition}

\begin{corollary}
  \label{obsDb}
Suppose that $\kappa$ is strongly inaccessible, $b\in {}^\kappa(\kappa+1)$
and $\bar{2}\leq b$. 
\begin{enumerate}
  \item Each $F^b_\alpha$ is a ${\leq}|\alpha|$--complete filter on $b(\alpha)$.
  \item If $b\in\bairek$ and $d(\alpha)=\big|\prod\limits_{\xi\leq \alpha}
    b(\alpha) \big|$, then the forcing notion $\bbD_b$ is reasonably
    A--bounding over $d$.
  \item $\bbD_b$ is purely sequentially$^+$ proper over semi diamonds. 
  \end{enumerate}
\end{corollary}

\begin{remark}
  The forcing notions $\bbD_b$ are natural generalizations of the forcings
  $\bbD_{\cX}$ studied in \cite{NeRo93}. This class includes such classical
  forcings for $\bairek$ as Kanamori's $\kappa$--Sacks forcing \cite{Ka80}
  or the forcing $\bbD_\kappa$ of \cite{RoSh:655}.
\end{remark}

\begin{proposition}
  \label{stateq}
Suppose that $\kappa$ is strongly inaccessible, $b\in {}^\kappa(\kappa+1)$
and $\bar{2}\leq b$. Let $\name{W}$ be the canonical name for the generic
function in $\prod\limits_{\alpha<\kappa} b(\alpha)$ added by $\bbD_b$. Then
for each $x\in \prod\limits_{\alpha<\kappa} b(\alpha)$ we have 
\[\forces_{\bbD_b}\mbox{`` the set }\big\{\alpha<\kappa: x(\alpha)=
  \name{W}(\alpha)\big\} \mbox{ is stationary ''.}\]
\end{proposition}

\begin{proof}
  By Corollary \ref{obsDb} we know that cardinals $\leq\kappa^+$ are not
  collapsed in forcing with $\bbD_b$.  Suppose that $\name{\alpha}_i$ (for
  $i<\kappa$) are $\bbD_b$--names for ordinals below $\kappa$ such that 
  $\forces_{\bbD_b}\mbox{`` }\langle \name{\alpha}_i: i<\kappa\rangle$ is
  continuous strictly increasing ''.   Let $T\in \bbD_b$.

  Construct inductively a sequence $\langle T_n,\beta_n,i_n:n<\omega\rangle$
  such that for each $n<\omega$ we have 
  \begin{itemize}
  \item $\beta_n<\beta_{n+1}<\kappa$, $i_n<i_{n+1}<\kappa$,
  \item $T= T_0\leq_{\bbD_b} T_n\leq_{\bbD_b} T_{n+1}$,
    $\lh\big(\mrot(T_n)\big)<\beta_n<\lh\big(\mrot(T_{n+1})\big)$, and 
\item   $T_{n+1}\forces_{\bbD_b}\name{\alpha}_{i_n}=\beta_n$. 
\end{itemize}
Let $\beta^*=\sup\{\beta_n:n<\omega\}$ and $i^*=\sup\{i_n:n<\omega\}$. By
$(\oplus)$ of the proof of \ref{q2isAb} we know that
$T^*=\bigcap\limits_{n<\omega} T_n\in\bbD_b$ is a condition stronger than 
all $T_n$. Also, $\beta^*= \lh\big(\mrot(T^*)\big)$ and
$T^*\forces_{\bbD_b}$`` $\name{\alpha}_{i^*}=\beta^*$ ''. Take $t\in T^*$ of
length $\beta^*+1$ such that $t(\beta^*)= x(\beta^*)$ (remember
$\suc_{T^*}(\mrot(T^*))=b(\beta^*)$). Then $(T^*)_t\forces
\name{W}(\name{\alpha}_{i^*}) =x(\name{\alpha}_{i^*})$.
\end{proof}

\begin{theorem}
  \label{Dbthm}
    Assume $\kappa$ is a strongly inaccessible cardinal, $2^\kappa=\kappa^+$,
    $b\in\bairek$ and $\bar{2}\leq b$. Let $c(\alpha)=\Big(\big|
    \prod\limits_{\xi \leq\alpha} b(\xi)\big|^{|\alpha|}\Big)^+$ for
    $\alpha<\kappa$. Then there are strategically
    $({<}\kappa)$--complete, $\kappa$--proper, $\kappa^{++}$--cc forcing
    notions  $\bbP,\bbQ$ such that
    \[\begin{array}{ll}
\forces_{\bbP} \mbox{``}&\gd_\kappa^c(\in^*) =\kappa^{+}\ \mbox{ and }\
  \gd_\kappa^b(\cancel{=^{\rm stat}})=2^\kappa= \kappa^{++}\mbox{ '',}\\
\forces_{\bbQ} \mbox{``}&\gd_\kappa(\leq^*) =\kappa^{+}\ \mbox{ and }\
\big(\forall d\in\bairek\big)\big(d\geq\bar{2}\ \Rightarrow\ \gd_\kappa^d
        (\cancel{=^{\rm stat}})=2^\kappa= \kappa^{++}\big)\mbox{ ''.}         
    \end{array}\]   
  \end{theorem}

  \begin{proof}
    For $\bbP$ consider $\kappa$--support iteration of $\bbD_b$ of length
    $\kappa^{++}$. The forcing $\bbQ$ can be built by iterating with
    $\kappa$--support all potential $\bbD_d$'s (in a manner similar to the
    proof of Theorem \ref{onebig}). 
  \end{proof}

The forcing notions $\bbD_k$ introduced in \cite{NeRo93} were used to
study the covering numbers of ideals $\cD_k$ associated with unsymmetric 
games. One of the key properties was their $k$--localization property. Since
the $k$--localization property of  \cite{NeRo93} was later applied in other
contexts (see Geschke and Quickert \cite{GeQu0x}, Geschke, Kojman, Kubi\'s
and Schipperus \cite{GKKS02},  Geschke and Kojman \cite{GeKo02} and 
Geschke \cite{Ge05}), it is tempting to generalize those localization
concepts to the case of iterations and/or products with
$\kappa$--support. Unfortunately, the iteration theorem for the most natural
generalizations fails miserably. 
  
  \begin{definition}
    \label{locdef}
    Let $b\in\cR^-_\kappa$.
    \begin{enumerate}
\item A {\em $(b,\kappa)$--localizing tree\/} is a complete $\kappa$--tree
  $\cS$ such that $|\suc_\cS(s)|<b\big(\lh(s)\big)$ for each $s\in\cS$.
\item A forcing notion $\bbP$ has {\em the $(b,\kappa)$--localization
    property\/} if
  \[
  \begin{array}{rl}
    \forces_\bbP\mbox{``}&\mbox{for every function }f:\kappa\longrightarrow
                           {\rm ON}\\
    &\mbox{there is a $(b,\kappa)$--localizing tree $\cS\in\bV$ such that
      $f\in\lim(\cS)$ ''.}
  \end{array}\]
    \end{enumerate}
  \end{definition}

  \begin{observation}
    \label{obsloc}
    Let $b\in\cR^-_\kappa$ and let $c(\alpha)=\big|\prod\limits_{\xi
      \leq\alpha} b(\xi)\big|^+$ for $\alpha<\kappa$.    
    \begin{enumerate}
    \item If $\bbP$ has the $(b,\kappa)$--localization property, then
      \[\forces_\bbP\mbox{`` }\big(\forall x\in\bairek\big)\big( \exists
        \varphi\in {\rm Loc}^{\bar{\kappa},c}_\kappa\cap\bV\big) \big(x\in^*
        \varphi\big) \mbox{ ''.}\]
\item If $\bbP$ is reasonably A--bounding over $b$, then it has the
  $(b,\kappa)$--localization property.      
    \end{enumerate}
  \end{observation}

  \begin{proposition}
\label{exloc}    
\begin{enumerate}
\item The forcing $\bbD_b$ does not have the $(b,\kappa)$--localization
  property.    
\item If $f,\bar{F}$ are as in Definition \ref{boundedPQ} and
  $|f(\alpha)|<c(\alpha)$ for every $\alpha<\kappa$, then the forcing notion
  $\bbQ^2_{f,\bar{F}}$ has the $(c,\kappa)$--localization property. In particular,
  $\bbD_b$ has the $(b^+,\kappa)$--localization property. 
\end{enumerate}
\end{proposition}
  
  \begin{proof}
    (1)\quad Suppose $\cS$ is a $(b,\kappa)$--localizing tree, $T\in \bbD_b$. Let
    $\alpha=\lh\big(\mrot(T)\big)$. Now either $\mrot(T)\notin \cS$ (so
    $T\forces \name{W}\notin \lim\big(\cS\big)$) or else
    $\big|\suc_\cS\big(\mrot(T)\big)\big|<b(\alpha)$, so we may pick $\xi\in
    b(\alpha)=\suc_T\big(\mrot(T)\big)$ such that $t=\mrot(T)\conc \langle\xi
    \rangle\notin \cS$. In the latter case, $(T)_t\forces
    \name{W}\notin\lim(\cS)$.  
    \medskip

    \noindent (2)\quad   Although this proposition appears trivial and
    follow from arguments for Proposition \ref{q2isAb}(3), it contains a
    subtle point that requires careful handling. We explain below why a
    naive approach does not work, which will motivate the need for
    additional conditions. 

  \noindent {\em An approach that does not work:}\quad Let $\name{g}$ be a
  $\bbQ^2_{f,\bar{F}}$--name for a function with domain $\kappa$ and values in
  ordinals. Let $R\in \bbQ^2_{f,\bar{F}}$. Consider a play $\langle
  (I_\alpha, \bar{S}^\alpha, \bar{T}^\alpha):\alpha<\kappa\rangle$ of
  $\Agame(R,\bbQ^2_{f,\bar{F}})$  where Generic uses her winning  strategy
  described in \ref{q2isAb}(3)   and Antigeneric moves as follows. 
  \begin{enumerate}
  \item[$(\spadesuit)$] At stage $\alpha<\kappa$ of the play, after
    Generic put  forward $I_\alpha=R_\alpha\cap \prod\limits_{\xi\leq\alpha}
    f(\xi)$ and $\bar{S}^\alpha=\langle S^\alpha_t: t\in I_\alpha \rangle$, 
    Antigeneric chooses conditions $T^\alpha_t\geq S^\alpha_t$ (for $t\in
    I_\alpha$) deciding the value of $\name{g}(\alpha)$. Say,
    $S^\alpha_t\forces  \name{g}(\alpha)=\Phi(\alpha,t)$ for some ordinals
    $\Phi(\alpha,t)$.  
  \end{enumerate}
After the play is over, a $c$--localizing tree
$R^*=\bigcap\limits_{\alpha<\kappa} R_\alpha$ has been constructed and it is
a condition witnessing that Generic won  $\Agame(R,\bbQ^2_{f,\bar{F}})$. For
every $\eta\in R^*$ we have a sequence $\Phi^*(\eta)=\langle \Phi(\alpha,
\eta\rest (\alpha+1)): \alpha+1\leq \lh(\eta)\rangle$. This produces a tree $\cS=
\{\Phi^*(\eta): \eta\in R\}$ such that $R^*\forces_{\bbQ^2_{f,\bar{F}}}
\name{g}\in \lim\big(\cS\big)$. However, we cannot claim that $\cS$ is a
$c$--localizing tree --- it may happen that $\Phi(\alpha,\eta)=0$ whenever
$\alpha<\alpha_0$ and then $\Phi(\alpha_0,\cdot)$ takes more than
$c(\alpha_0)$ distinct values. [Note that this would not be a problem if $c$
was increasing fast enough; cf Observation \ref{obsloc}(2).]
\medskip

The key observation needed for a complete proof is the following general
claim.

\begin{claim}
  \label{cl9}
Suppose $\bbP$ is a $({<}\kappa)$--strategically complete forcing notion, and 
$\name{\tau}$ is a $\bbP$--name for a function from $\kappa$ to
ordinals. Assume $p\forces_\bbP \name{\tau}\notin \bV$ and let
$\alpha<\kappa$ and $\mu<\kappa$. Then there are $\beta>\alpha$ and pairwise
different functions $\sigma_\xi$ on $[\alpha,\beta]$ and conditions
$p_\xi\geq p$ (for $\xi<\mu$)  such that $p_\xi\forces_\bbP \name{\tau}\rest
[\alpha,\beta]=\sigma_\xi$ for all $\xi<\mu$. 
\end{claim}

\begin{proof}[Proof of the Claim]
  Since $\bbP$ is a $({<}\kappa)$--strategically complete, for each
  $\alpha'>\alpha$ we have $\forces \name{\tau}\rest
  [\alpha,\alpha']\in\bV$. Since $p\forces \name{\tau}\notin \bV$, for every
  $q\geq p$ and $\alpha'>\alpha$ there are $q_0,q_1\geq q$,
  $\alpha''>\alpha'$ and $\sigma_0\neq\sigma_1$ such that $q_\ell\forces
  \name{\tau} \rest [\alpha',\alpha'']=\sigma_\ell$ for $\ell=0,1$. Using
  the strategic completeness we may iterate this process to get the
  conclusion of the Claim. 
\end{proof}

Now we slightly modify our previous (failed) approach. We start with a name
$\name{g}$ for a function and assume $R\forces_{\bbQ^2_{f,\bar{F}}}
\name{g}\notin \bV$ (otherwise nothing to do). We consider a play $\langle
  (I_\alpha, \bar{S}^\alpha, \bar{T}^\alpha):\alpha<\kappa\rangle$ of
  $\Agame(R,\bbQ^2_{f,\bar{F}})$  in which Generic uses her winning  strategy
  presented in \ref{q2isAb}(3)   and Antigeneric moves as follows.

  During the course of the play, Antigeneric writes aside a continuous
  increasing sequence of ordinals $\langle \gamma_\alpha:
  \alpha<\kappa\rangle$,  $\gamma_0=0$, and systems $\langle \sigma^i_t:t\in
  I_\alpha\rangle$ (for $\alpha<\kappa$). 
  
  \begin{enumerate}
  \item[$(\heartsuit)$] At stage $\alpha<\kappa$ of the play, after
    Generic put  forward $I_\alpha=R_\alpha\cap \prod\limits_{\xi\leq\alpha}
    f(\xi)$ and $\bar{S}^\alpha=\langle S^\alpha_t: t\in I_\alpha \rangle$, 
    Antigeneric uses Claim \ref{cl9} to choose  $\gamma_{\alpha+1}>
    \gamma_\alpha$  and \underline{pairwise distinct} functions
    $\sigma^\alpha_t: [\gamma_\alpha,\gamma_{\alpha+1})\longrightarrow {\rm
      ON}$ and conditions $T^\alpha_t\geq S^\alpha_t$ (for $t\in
    I_\alpha$) such that  $T^\alpha_t\forces  \name{g}\rest [\gamma_\alpha,
    \gamma_{\alpha+1}) =\sigma^\alpha_t$.
  \end{enumerate}
After the play is over, a $c$--localizing tree
$R^*=\bigcap\limits_{\alpha<\kappa} R_\alpha$ has been constructed and it is
a condition witnessing that Generic won  $\Agame(R,\bbQ^2_{f,\bar{F}})$. For
every $\eta\in R^*$ we have a sequence $\Phi^*(\eta)=\bigcup\big\{
\sigma^\alpha_{\eta\rest (\alpha+1)}: \alpha+1\leq \lh(\eta)\big\}$. This
produces a tree $\cS= \{\Phi^*(\eta)\rest\xi: \eta\in R,\
\xi<\gamma_{\lh(\eta)}\}$ such that $R^*\forces_{\bbQ^2_{f,\bar{F}}} 
\name{g}\in \lim\big(\cS\big)$. By our construction, we may also claim that
$\cS$ is a $c$--localizing tree.
  \end{proof}

Unlike the case of $k$--localizations in \cite{NeRo93},
$(b,\kappa)$--localization property is not preserved in $\kappa$--support
iterations and/or products. Let us give an appropriate example for a product.

\begin{example}
  Let $2<\mu<\kappa$ and let $\lambda$ be an infinite cardinal, $\mu\leq
  \lambda< \kappa$ and let $\mu^*=\mu^\lambda$. Let $\bbQ$ be the (full
  support) product of $\lambda$ copies of $\bbD_{\bar{\mu}}$. Then $\bbQ$
  fails the $(\overline{\mu^*},\kappa)$--localization property (so it also fails
  $(\overline{\mu^+},\kappa)$--localization). 
\end{example}

\begin{proof}
Fix a bijection $\upsilon:{}^\lambda\mu\longrightarrow \mu^*$  and define
$\Upsilon:{}^\lambda\big({}^\kappa\mu)\longrightarrow
{}^\kappa\big(\mu^*\big)$ by letting for $\eta_\alpha\in {}^\kappa\mu$ (for 
$\alpha<\lambda$): 
  \[\Upsilon\Big(\big\langle\eta_\alpha:\alpha<\lambda\big\rangle \Big)=
\Big\langle \upsilon\big(\langle\eta_\alpha(\xi):\alpha<\lambda\rangle
\big): \xi<\kappa \Big\rangle.\]
Let $\name{W}_\alpha$ be a $\bbQ$--name for the generic function in
${}^\kappa\mu$ added by $\bbD_{\bar{\mu}}$ on $\alpha^{\rm th}$ coordinate
and let $\name{V}$ be a $\bbQ$--name for $\Upsilon\big(\langle
\name{W}_\alpha: \alpha<\lambda\rangle\big)$.

Assume $\cS$ is a $(\overline{\mu^*},\kappa)$--localizing tree and
$\bar{S}=\langle S_\alpha:\alpha<\lambda\rangle\in\bbQ$. It should be clear
that we may choose $\xi<\kappa$ and a condition $\bar{T}=\langle
T_\alpha:\alpha<\lambda\rangle\in\bbQ$ stronger than $\bar{S}$ and such that
$\lh\big(\mrot(T_\alpha)\big)=\xi$ for all $\alpha<\lambda$. Let
\[s=\big\langle \upsilon\big(\langle\mrot(T_\alpha)(\zeta):\alpha<
  \lambda\rangle \big): \zeta<\xi \big\rangle\in {}^\xi(\mu^*).\]
Now, if $s\notin \cS$, then $\bar{T}\forces_{\bbQ} \name{V}\notin
\lim\big(\cS\big)$. If $s\in \cS$, then $|\suc_{\cS}(s)|<\mu^*$ so we may
find $\sigma=\langle\sigma_\alpha:\alpha<\lambda\rangle\in {}^\lambda\mu$
such that $\upsilon(\sigma)\notin\suc_\cS(s)$. Put $t_\alpha
=\mrot(T_\alpha)\conc \langle \sigma_\alpha\rangle$ for $\alpha<\lambda$ and
note that $t_\alpha\in T_\alpha$. Finally, let $\bar{T}^*=\langle
\big(T_\alpha\big)_{t_\alpha}:\alpha<\lambda\rangle$. It should be clear that
$\bar{T}^* \in\bbQ$ is stronger than $\bar{T}$ so also stronger than
$\bar{S}$ and $\bar{T}^*\forces_\bbQ \name{V}\notin \lim(\cS)$. 
\end{proof}

  

\end{document}